\documentclass[nointlimits,11pt,oneside]{amsart}
\usepackage{amssymb,cases,enumitem, verbatim}
\usepackage{xcolor}
\usepackage{hyperref}
\usepackage[T1]{fontenc}
\usepackage[utf8]{inputenc}

\hypersetup{
	colorlinks=true,
	linkcolor=blue,
	citecolor=blue,
	filecolor=green,
	urlcolor=cyan,
	bookmarks=true
}

\makeatletter
\renewcommand*{\eqref}[1]{%
	\hyperref[{#1}]{\textup{\tagform@{\ref*{#1}}}}%
}
\makeatother

\setlist[enumerate,1]{label={\textup{(\roman*)}}}

\usepackage[%
	a4paper,
	total={16cm,23cm},
	left=3cm, top=3cm,
	marginparsep=2pt]
{geometry}

\theoremstyle{plain}
\newtheorem{theorem}{Theorem}[section]
\newtheorem{lemma}[theorem]{Lemma}
\newtheorem{corollary}[theorem]{Corollary}
\newtheorem{proposition}[theorem]{Proposition}

\newtheorem{convention}[theorem]{Convention}
\theoremstyle{definition}
\newtheorem{remark}[theorem]{Remark}

\newtheorem{definition}[theorem]{Definition}
\newtheorem{example}[theorem]{Example}

\numberwithin{equation}{section}

\def\L1loc{L^1_{\text{loc}}}

\begin{document}

\title{On the properties of rearrangement-invariant quasi-Banach function spaces}
\author{Anna Musilov{\' a}, Ale\v s Nekvinda, Dalimil Pe\v sa, and Hana Tur{\v c}inov{\' a}}

\address{Anna Musilov{\' a}, Department of Mathematical Analysis, Faculty of Mathematics and
	Physics, Charles University, Sokolovsk\'a~83,
	186~75 Praha~8, Czech Republic}
\email{anna.musilova@mff.cuni.cz}
\urladdr{0009-0005-2600-5082}

\address{Ale\v s Nekvinda, Department of Mathematics, Faculty of Civil Engineereng,
	Czech Technical University, Th\' akurova 7, 16629 Prague~6,
	Czech Republic}
\email{ales.nekvinda@cvut.cz}
\urladdr{0000-0001-6303-5882}

\address{Dalimil Pe{\v s}a, Department of Mathematical Analysis, Faculty of Mathematics and
	Physics, Charles University, Sokolovsk\'a~83,
	186~75 Praha~8, Czech Republic}
\email{pesa@karlin.mff.cuni.cz}
\urladdr{0000-0001-6638-0913}

\address{Hana Tur{\v c}inov{\' a}, Department of Mathematical Analysis, Faculty of Mathematics and Physics, Charles University, Sokolovsk{\' a} 83, 186 75 Praha 8, Czech Republic
	and Department of Mathematics, Faculty of Electrical Engineering, Czech Technical University in Prague, Technick{\' a} 2, 166 27 Praha 6, Czech Republic}
\email{turcinova@karlin.mff.cuni.cz}
\urladdr{0000-0002-5424-9413}


\thanks{The first and fourth authors have been supported by the Primus research programme PRIMUS/21/SCI/002 of Charles University. The second, third, and fourth authors have been supported by the grant 23-04720S of the Czech Science Foundation. The third author has been further supported by the Charles University Grant Agency, project No.~234522, the grant 24-10505S of the Czech Science Foundation, and Charles University Research program No.~UNCE24/SCI/005. The fourth author has been further supported by the grant 21-01976S of the Czech Science Foundation and the Charles University Research program No.~UNCE/SCI/023.}

\begin{abstract}
This paper explores some important aspects of the theory of rearrangement-invariant quasi-Banach function spaces. We focus on two main topics. Firstly, we prove an analogue of the Luxemburg representation theorem for rearrangement-invariant quasi-Banach function spaces over resonant measure spaces. Secondly, we develop the theory of fundamental functions and endpoint spaces.
\end{abstract}

\date{\today}

\maketitle

\makeatletter
   \providecommand\@dotsep{2}
\makeatother

\section{Introduction}
The class of quasi-Banach function spaces has attracted a decent amount of attraction in the recent years as evidenced for example by the papers \cite{Baena-MiretGogatishvili22}, \cite{CaetanoGogatishvili16}, \cite{EdmundsKerman00}, \cite{GuoZhao20}, \cite{HaleNaibo23}, \cite{Haroske07Notes}, \cite{Ho16}, \cite{Ho20}, \cite{Ho20-Interpolation}, \cite{Ho23}, \cite{Kolwicz18}, \cite{KolwiczLesnik19}, \cite{LoristNieraeth23}, \cite{MalyL12}, \cite{MalyL16}, \cite{NekvindaPesa20}, \cite{Nieraeth23}, \cite{PanYang23}, \cite{Pesa22}, \cite{SawanoHo17}, \cite{SunYang22}, \cite{WangYang20}, and \cite{WangYang23} (see Introduction of \cite{NekvindaPesa20} for a more detailed overview). The reason for this is probably that the class is more general than the more classical abstract classes of function spaces, e.g.~the Banach function spaces, and it includes many quasinormed spaces that often appear naturally in estimates for important operators. The most natural example of such space is $L^{1, \infty}$, also known as weak $L^1$, that appears in the endpoint estimates of the Hardy--Littlewood maximal operator and many other important operators in harmonic analysis. On the other hand, the class is still quite manageable and many of the useful tools available in the more classical settings are still present. Hence, the class is proving to be a good setting for developing many of the theories that were previously restricted to the more classical and smaller classes.

In this paper, we aim to continue the work of \cite{NekvindaPesa20} where we started developing the general theory of quasi-Banach function spaces and provided several useful tools that should make their handling significantly easier. This time, we focus solely on the subclass of rearrangement-invariant quasi-Banach function spaces and cover two important topics. First is the Luxemburg representation theorem, which is a cornerstone in the classical theory of rearrangement-invariant Banach function spaces. As illustrated below, the absence of this result for the wider class of rearrangement-invariant quasi-Banach function spaces was a gaping hole in the theory which caused many problems. Thus the extension we provide should prove very useful and allow the theory to further develop and expand its applications. The second topic we cover is the theory related to the concepts of fundamental function and endpoint spaces, an area that also has many significant applications in the classical theory.

Let us now discuss the two above mentioned topics in more detail. The first is an extension to the class of r.i.~quasi-Banach function spaces of the fundamental and extremely important Luxemburg representation theorem. In its classical form, this result states that to every rearrangement-invariant Banach function norm $\lVert \cdot \lVert_X$ over an arbitrary resonant measure space corresponds a norm, traditionally denoted $\lVert \cdot \lVert_{\overline{X}}$, which is of the same class but acting over the appropriate subset of $[0, \infty)$ equipped with the classical Lebesgue measure and which satisfies for every measurable function $f$ defined on the original measure space that
\begin{equation*}
	\lVert f \lVert_X = \lVert f^* \lVert_{\overline{X}},
\end{equation*}
where $f^*$ is the non-increasing rearrangement of $f$ (see Theorem~\ref{TheoremLuxemburgRepresentation} for the precise formulation and Sections~\ref{SectionNon-increasingRearrangement} and \ref{SectionFunctionNormsQuasinorms} for the relevant definitions). 

This result was first obtained by Luxemburg in \cite{Luxemburg67} (with \cite[Chapter~2, Theorem~4.10]{BennettSharpley88} serving as the standard modern reference) and has since proved to be a crucial component in many aspects of the theory of rearrangement-invariant Banach function spaces. Its importance lies in that it allows one to work with the easily manageable functions defined on an interval, which makes possible many constructions that cannot be done on abstract measure spaces. Its many applications include the construction of optimal targets and domains for Sobolev embeddings such as in \cite{AlbericoCianchi2018}, \cite{BreitCianchi21}, \cite{CianchiPick15}, \cite{Mihula21}, \cite{Mihula21Poincare}, \cite{Musil18}, \cite{MusilPick23}, and \cite{Vybiral07}; Gagliardo--Nirenberg inequalities and their applications as in \cite{FiorenzaFormica19}, \cite{LesnikRoskovec23}, and \cite{MolchanovaRoskovec21}; compactness of Sobolev embeddings as in \cite{Slavikova15}; description of optimal domains and targets for the boundedness of many classical operators as in \cite{CianchiMusil23}, \cite{EdmundsMihula20}, and \cite{MusilPick23}; reducing estimates for $K$-functionals to one-dimensional inequalities as in \cite{Baena-MiretGogatishvili22}; Korn-type inequalities as in \cite{BreitCianchi17}; constructing function spaces in an abstract manner as in \cite{Pesa22} and \cite{Turcinova23}.

The fact that the Luxemburg representation theorem has until now been restricted to the normed case has proved to be an obstacle encountered by many, and some partial or ad hoc solutions have been developed. For example, in \cite{Baena-MiretGogatishvili22}, \cite{Ho16}, \cite{Ho20}, \cite{Ho20-Interpolation}, and \cite{Ho23} the issue was sidestepped by assuming a~priori that the r.i.~quasi-Banach function norm in question, acting over more general measure space, is derived from some r.i.~quasi-Banach function norm defined over $(0, \infty)$, which then obviously served as the needed representation. Another approach can be used when the focus is on a more specific class of function spaces for which the quasinorm is given by a formula that is independent of the underlying measure space, as then the space over $(0, \infty)$ that is constructed using the same formula typically provides the required representation; this approach has been successfully applied e.g.~in \cite{KrepelaMihula22}, \cite{PesaLK}, and \cite[Section~6]{Turcinova23}. On the other hand, when one wishes to work with r.i.~quasi-Banach function spaces in abstract, the lack of Luxemburg representation theorem has proven to be an insurmountable obstacle, forcing authors to avoid this issue by working only in restricted setting; this was the case of e.g.~\cite{Pesa22}, \cite{Turcinova23}.

This topic is the content of Section~\ref{SectionRepresentation} where we prove in Theorem~\ref{TheoremRepresentation} that the Luxemburg representation theorem holds even in the wider class of r.i.~quasi-Banach function spaces over any resonant measure space, although the properties of the representation functor, so to say, are different for non-atomic and completely atomic measure spaces (as is to be expected, given the same is true for the classical normed case). Needless to say, the classical proof as presented e.g.~in \cite[Chapter~2, Theorem~4.10]{BennettSharpley88} cannot be used in this wider setting, as it is based on the Köthe reflexivity of Banach function spaces, i.e.~the fact that every Banach function space $X$ satisfies $X = X''$ with equal norms (with $X''$ being the second associate space, also called the second Köthe dual, of $X$), which is in fact a characterisation of Banach function spaces. Hence, the proof had to be developed from scratch and it turned out that each of the cases of resonant measure spaces (i.e.~non-atomic and completely atomic spaces) required their own distinct construction with unique challenges.

The second topic of this paper is the study of the concepts of fundamental function and endpoint spaces in the context of r.i.~quasi-Banach function spaces. This theory is classical for r.i.~Banach function spaces and has also proven to be extremely useful. For example, it has been employed to characterise the optimal Sobolev embeddings for the class of Orlicz spaces in \cite{CianchiMusil19} and \cite{Musil18} and also for more general optimality problems for Orlicz spaces in \cite{MusilPick23}. It is related to the validity of Gagliardo--Nirenberg inequalities, see e.g.~\cite{FiorenzaFormica19}. It plays an important role in the research of optimal embeddings of generalised Besov spaces, see e.g.~\cite{BashirKaradzhov11} or \cite{MartinMilman06}. It has been studied in relation to various questions of the theories of interpolation and extrapolation in \cite{Berezhnoi13}, \cite{CobosFernandez-Cabrera17}, \cite{CwikelNilsson85}, \cite{CwikelPustylnik00}, \cite{Haroske07}, \cite{Maligranda83}, \cite{MalyL12Interpolation}, \cite{MusilOlhava17}, \cite{Pustylnik01}, and \cite{Sharpley71}. This theory is also related to the study of properties of embeddings, such as it being almost compact (see \cite{Fernandez-MartinezManzano10} and \cite{Slavikova12}) or maximally non-compact (see \cite{MalyMihulaUNP}). Other areas of applications include generalised duality as in \cite{MaligrandaPersson89}, problems related to some types of PDEs as in \cite{CaicedoCuevas21}, and various properties of Banach spaces as in \cite{AstashkinSukochev07} or \cite{dePagterSukochev20}. The importance of this topic is further illustrated by the fact that it is included in most of the books that present the theory of rearrangement-invariant Banach function spaces, e.g.~\cite{BennettSharpley88}, \cite{KreinPetunin82}, and \cite{FucikKufner13}.

Section~\ref{SectionFundFuncResearch} is devoted to this topic, as we characterise which functions serve as a fundamental function of some r.i.~quasi-Banach function space, show that for each such function there corresponds a space, which we call the weak Marcinkiewicz endpoint space, that is the largest r.i.~quasi-Banach function space with this fundamental function. This space is not new, as it is closely related to the classical Marcinkiewicz endpoint space, see e.g.~\cite[Section~7.10]{FucikKufner13}, and so its natural definition has appeared before, but always as a side note to the study of the more classical endpoint spaces. Consequently, the properties of this type of spaces were so far studied neither deeply nor comprehensively, and many interesting questions have therefore been open until now. We rectify this in Section~\ref{SectionFundFuncResearch}, and also provide some results showing what can be inferred about a space from the properties of its fundamental function and how the concept of fundamental function interacts with that of the associate space.

\section{Preliminaries}

The purpose of this section is to establish the theoretical background that serves as the foundation for our research. The definitions and notation is intended to be as standard as possible. The usual reference for most of this theory is \cite{BennettSharpley88}.

Throughout this paper we will denote by $(\mathcal{R}, \mu)$, and occasionally by $(\mathcal{S}, \nu)$, some arbitrary (totally) $\sigma$-finite measure space. Given a $\mu$-measurable set $E \subseteq \mathcal{R}$ we will denote its characteristic function by $\chi_E$. By $\mathcal{M}(\mathcal{R}, \mu)$ we will denote the set of all extended complex-valued $\mu$-measurable functions defined on $\mathcal{R}$. As is customary, we will identify functions that coincide $\mu$-almost everywhere. We will further denote by $\mathcal{M}_0(\mathcal{R}, \mu)$ and $\mathcal{M}_+(\mathcal{R}, \mu)$ the subsets of $\mathcal{M}(\mathcal{R}, \mu)$ containing, respectively, the functions finite $\mu$-almost everywhere and the non-negative functions.

When there is no risk of confusing the reader, we will abbreviate $\mu$-almost everywhere, $\mathcal{M}(\mathcal{R}, \mu)$, $\mathcal{M}_0(\mathcal{R}, \mu)$, and $\mathcal{M}_+(\mathcal{R}, \mu)$ to $\mu$-a.e., $\mathcal{M}$, $\mathcal{M}_0$, and $\mathcal{M}_+$, respectively.
 
When $X$ is a set and $f, g: X \to \mathbb{C}$ are two maps satisfying that there is some positive and finite constant $C$, depending only on $f$ and $g$, such that $\lvert f(x) \rvert \leq C \lvert g(x) \rvert$ for all $x \in X$, we will denote this by $f \lesssim g$. We will also write $f \approx g$, or sometimes say that $f$ and $g$ are equivalent, whenever both $f \lesssim g$ and $g \lesssim f$ are true at the same time. We choose this general definition because we will use the symbols ``$\lesssim$'' and ``$\approx$'' with both functions and functionals. 

When $X, Y$ are two topological linear spaces, we will denote by $Y \hookrightarrow X$ that $Y \subseteq X$ and that the identity mapping $I : Y \rightarrow X$ is continuous.

As for some special cases, we will denote by $\lambda^n$ the classical $n$-dimensional Lebesgue measure, with the exception of the $1$-dimensional case where we simply write $\lambda$. We will further denote by $m$ the counting measure over $\mathbb{N}$. When $p \in (0, \infty]$ we will denote by $L^p$ the classical Lebesgue space (of functions in $\mathcal{M}(\mathcal{R}, \mu)$), defined for finite $p$ as the set
\begin{equation*}
	L^p = \left \{ f \in \mathcal{M}(\mathcal{R}, \mu); \; \int_{\mathcal{R}} \lvert f \rvert^p \: d\mu < \infty \right \},
\end{equation*}
equipped with the customary (quasi-)norm
\begin{equation*}
	\lVert f \rVert_p = \left ( \int_ {\mathcal{R}} \lvert f \rvert^p \: d\mu \right )^{\frac{1}{p} },
\end{equation*}
and through the usual modifications for $p=\infty$. In the special case when $(\mathcal{R}, \mu) = (\mathbb{N}, m)$ we will denote this space by $l^p$. Note that in this paper we consider $0$ to be an element of $\mathbb{N}$.

We will work extensively with the $\Delta_2$-condition.

\begin{definition}
	Let $\varphi: [0, \infty) \to [0, \infty)$ be non-decreasing. We say that $\varphi$ satisfies the $\Delta_2$-condition if there exists a constant $C_{\varphi} \in [1, \infty)$ such that it holds for all $t \in [0, \infty)$ that
	\begin{equation*}
		\varphi(2t) \leq C_{\varphi} \varphi(t).
	\end{equation*}
\end{definition}

The following statement is a simple consequence of the definition; we list it here together with its proof for the sake of completeness and since it appears that the relationship is not that well known.

\begin{proposition} \label{PropositionD2ImpliesSubadditivity}
	Let $\varphi: [0, \infty) \to [0, \infty)$ be a non-decreasing function that satisfies the $\Delta_2$-condition with constant $C_{\varphi}$. Then it is also subadditive up to this constant, i.e.~it holds for every $s,t \in [0, \infty)$ that
	\begin{equation*}
		\varphi(s+t) \leq C_{\varphi} (\varphi(s) + \varphi(t)).
	\end{equation*}
\end{proposition}

\begin{proof}
	Let $s,t \in [0, \infty)$ and assume without loss of generality that $s \leq t$, i.e.~$\frac{s+t}{2} \leq t$. We now use the $\Delta_2$-condition, monotonicity, and non-negativity of $\varphi$ to compute
	\begin{equation*}
		\varphi(s+t) \leq C_{\varphi} \varphi \left (\frac{s + t}{2} \right ) \leq C_{\varphi} \varphi(t) \leq C_{\varphi} (\varphi(s) + \varphi(t)).
	\end{equation*}
\end{proof}

We will also use the following easy observation:

\begin{proposition} \label{PropositionD2Max}
	Let $\varphi_1: [0, \infty) \to [0, \infty)$ and $\varphi_2: [0, \infty) \to [0, \infty)$ be non-decreasing functions that satisfy the $\Delta_2$-condition with constants $C_{\varphi_1}$ and $C_{\varphi_2}$, respectively. Then the function $\varphi: [0, \infty) \to [0, \infty)$ defined for $t \in [0, \infty)$ by
	\begin{equation*}
		\varphi(t) = \max \{ \varphi_1(t), \varphi_2(t)\}
	\end{equation*}
	also satisfies the $\Delta_2$-condition with the constant $C_{\varphi} = \max \{C_{\varphi_1}, C_{\varphi_2} \}$.
\end{proposition}

\subsection{Non-increasing rearrangement} \label{SectionNon-increasingRearrangement}
We now present the crucial concept of the non-increasing rearrangement of a function and state some of its properties that will be important for our work. We proceed in accordance with \cite[Chapter~2]{BennettSharpley88}.

\begin{definition}	
	The distribution function $f_*$ of a function $f \in \mathcal{M}$ is defined for $s \in [0, \infty)$ by
	\begin{equation*}
		f_*(s) = \mu(\{ t \in \mathcal{R}; \; \lvert f(t) \rvert > s \}).
	\end{equation*}	
	The non-increasing rearrangement $f^*$ of the said function is then defined for $t \in [0, \infty)$ by
	\begin{equation*}
		f^*(t) = \inf \{ s \in [0, \infty); \; f_*(s) \leq t \}.
	\end{equation*}
\end{definition}

For the basic properties of the distribution function and the non-increasing rearrangement, with proofs, see \cite[Chapter~2, Proposition~1.3]{BennettSharpley88} and \cite[Chapter~2, Proposition~1.7]{BennettSharpley88}, respectively. We consider those properties to be classical and well known and we will be using them without further explicit reference; let us, however, just point out one critical property we will use several times in the paper, that is that we have for all $f,g \in \mathcal{M}$ and all $t \in [0, \infty)$ that
\begin{equation*}
	(f+g)^* (t) \leq f^* \left( \frac{t}{2} \right) + g^* \left( \frac{t}{2} \right).
\end{equation*}

An important concept used in the paper is that of equimeasurability.

\begin{definition} \label{DEM}
	We say that the functions $f \in \mathcal{M}(\mathcal{R}, \mu)$ and $g \in \mathcal{M}(\mathcal{S}, \nu)$ are equimeasurable if $f^* = g^*$, where the non-increasing rearrangements are computed with respect to the appropriate measures.
\end{definition}

It is not hard to show that two functions are equimeasurable if and only if their distribution functions coincide too (with each distribution function being considered with respect to the appropriate measure).

A very significant classical result is the Hardy--Littlewood inequality. For proof, see for example \cite[Chapter~2, Theorem~2.2]{BennettSharpley88}.

\begin{theorem} \label{THLI}
	It holds for all $f, g \in \mathcal{M}$ that
	\begin{equation*}
		\int_\mathcal{R} \lvert fg \rvert \: d\mu \leq \int_0^{\infty} f^*g^* \: d\lambda.
	\end{equation*}
\end{theorem}

It follows directly from this result that it holds for every $f,g \in \mathcal{M}$ that
\begin{equation} \label{HLI_sup}
	\sup_{\substack{\tilde{g} \in \mathcal{M} \\ \tilde{g}^* = g^*}} \int_{\mathcal{R}} \lvert f \tilde{g} \rvert \: d\mu \leq \int_0^{\infty} f^*g^* \: d\lambda.
\end{equation}
This motivates the definition of resonant measure spaces as those spaces where we have equality in \eqref{HLI_sup}.

\begin{definition}
	A $\sigma$-finite measure space $(\mathcal{R}, \mu)$ is said to be resonant if it holds for all $f, g \in \mathcal{M}(\mathcal{R}, \mu)$ that
	\begin{equation*}
		\sup_{\substack{\tilde{g} \in \mathcal{M} \\ \tilde{g}^* = g^*}} \int_\mathcal{R} \lvert f \tilde{g} \rvert \: d\mu = \int_0^{\infty} f^* g^* \: d\lambda.
	\end{equation*}
\end{definition}

The property of being resonant is a crucial one. Luckily, there is a straightforward characterisation of resonant measure spaces which we present below. For proof and further details see \cite[Chapter~2, Theorem~2.7]{BennettSharpley88}.

\begin{theorem} \label{TheoremCharResonance}
	A $\sigma$-finite measure space is resonant if and only if it is either non-atomic or completely atomic with all atoms having equal measure.
\end{theorem}

Finally, we will also need the concept of elementary maximal function, sometimes also called the maximal non-increasing rearrangement, which is defined as the Hardy transform of the non-increasing rearrangement.

\begin{definition}
	The elementary maximal function $f^{**}$ of $f \in \mathcal{M}$ is defined for $t \in (0, \infty)$ by 
	\begin{equation*}
		f^{**}(t) = \frac{1}{t} \int_0^{t} f^*(s) \: ds.
	\end{equation*} 
\end{definition} 

\subsection{Banach function norms and quasinorms} \label{SectionFunctionNormsQuasinorms}

\begin{definition}
	Let $\lVert \cdot \rVert : \mathcal{M}(\mathcal{R}, \mu) \rightarrow [0, \infty]$ be a mapping satisfying $\lVert \, \lvert f \rvert \, \rVert = \lVert f \rVert$ for all $f \in \mathcal{M}$. We say that $\lVert \cdot \rVert$ is a Banach function norm if its restriction to $\mathcal{M}_+$ satisfies the following axioms:
	\begin{enumerate}[label=\textup{(P\arabic*)}, series=P]
		\item \label{P1} it is a norm, in the sense that it satisfies the following three conditions:
		\begin{enumerate}[ref=(\theenumii)]
			\item \label{P1a} it is positively homogeneous, i.e.\ $\forall a \in \mathbb{C} \; \forall f \in \mathcal{M}_+ : \lVert a f \rVert = \lvert a \rvert \lVert f \rVert$,
			\item \label{P1b} it satisfies $\lVert f \rVert = 0 \Leftrightarrow f = 0$  $\mu$-a.e.,
			\item \label{P1c} it is subadditive, i.e.\ $\forall f,g \in \mathcal{M}_+ \: : \: \lVert f+g \rVert \leq \lVert f \rVert + \lVert g \rVert$,
		\end{enumerate}
		\item \label{P2} it has the lattice property, i.e.\ if some $f, g \in \mathcal{M}_+$ satisfy $f \leq g$ $\mu$-a.e., then also $\lVert f \rVert \leq \lVert g \rVert$,
		\item \label{P3} it has the Fatou property, i.e.\ if  some $f_n, f \in \mathcal{M}_+$ satisfy $f_n \uparrow f$ $\mu$-a.e., then also $\lVert f_n \rVert \uparrow \lVert f \rVert $,
		\item \label{P4} $\lVert \chi_E \rVert < \infty$ for all $E \subseteq \mathcal{R}$ satisfying $\mu(E) < \infty$,
		\item \label{P5} for every $E \subseteq \mathcal{R}$ satisfying $\mu(E) < \infty$ there exists some finite constant $C_E$, dependent only on $E$, such that the inequality $ \int_E f \: d\mu \leq C_E \lVert f \rVert $ is true for all $f \in \mathcal{M}_+$.
	\end{enumerate} 
\end{definition}

\begin{definition}
	Let $\lVert \cdot \rVert : \mathcal{M}(\mathcal{R}, \mu) \rightarrow [0, \infty]$ be a mapping satisfying $\lVert \, \lvert f \rvert \, \rVert = \lVert f \rVert$ for all $f \in \mathcal{M}$. We say that $\lVert \cdot \rVert$ is a quasi-Banach function norm if its restriction to $\mathcal{M}_+$ satisfies the axioms \ref{P2}, \ref{P3} and \ref{P4} of Banach function norms together with a weaker version of axiom \ref{P1}, namely
	\begin{enumerate}[label=\textup{(Q\arabic*)}]
		\item \label{Q1} it is a quasinorm, in the sense that it satisfies the following three conditions:
		\begin{enumerate}[ref=(\theenumii)]
			\item \label{Q1a} it is positively homogeneous, i.e.\ $\forall a \in \mathbb{C} \; \forall f \in \mathcal{M}_+ : \lVert af \rVert = \lvert a \rvert \lVert f \rVert$,
			\item \label{Q1b} it satisfies  $\lVert f \rVert = 0 \Leftrightarrow f = 0$ $\mu$-a.e.,
			\item \label{Q1c} there is a constant $C\geq 1$, called the modulus of concavity of $\lVert \cdot \rVert$, such that it is subadditive up to this constant, i.e.
			\begin{equation*}
				\forall f,g \in \mathcal{M}_+ : \lVert f+g \rVert \leq C(\lVert f \rVert + \lVert g \rVert).
			\end{equation*}
		\end{enumerate}
	\end{enumerate}
\end{definition}

Usually, it is assumed that the modulus of concavity is the smallest constant for which the part \ref{Q1c} of \ref{Q1} holds. We will follow this convention, even though the value will be of little consequence for our results. 

\begin{definition}
	Let $\lVert \cdot \rVert_X$ be a (quasi-)Banach function norm. We say that $\lVert \cdot \rVert_X$ is rearrangement-invariant, abbreviated r.i., if $\lVert f\rVert_X = \lVert g \rVert_X$ whenever $f, g \in \mathcal{M}$ are equimeasurable (in the sense of Definition~\ref{DEM}).
\end{definition}

\begin{definition}
	Let $\lVert \cdot \rVert_X$ be a (quasi-)Banach function norm. We then define the corresponding (quasi-)Banach function space $X$ as the set
	\begin{equation*}
		X = \left \{ f \in \mathcal{M};  \; \lVert f \rVert_X < \infty \right \}.
	\end{equation*}
	
	Furthermore, we will say that $X$ is rearrangement-invariant whenever $\lVert \cdot \rVert_X$ is.
\end{definition}

Before we present an overview of the properties of the above mentioned spaces, let us give some examples. We focus on the most classical ones as well as those that we need in the paper.

\begin{example}
	The Lebesgue spaces $L^p$, as introduces above, are r.i.~Banach function spaces for $p \in [1, \infty]$ and r.i.~quasi-Banach function spaces for $p \in (0,1)$.
\end{example}

\begin{example}
	Let $p, q \in (0, \infty]$. The Lorentz spaces $L^{p,q}$, defined for finite $q$ via the functional
	\begin{equation*}
		\lVert f \rVert_{p,q} = \int_0^{\infty} t^{\frac{q}{p} - 1} (f^*(t))^q \: dt
	\end{equation*}
	and for infinite $q$ via the functional
	\begin{equation*}
		\lVert f \rVert_{p,q} = \sup_{t \in [0, \infty)} f^{\frac{1}{p}} f^*(t),
	\end{equation*}
	are r.i.~quasi-Banach function spaces if and only if either $p < \infty$ or both $p = \infty$ and $q = \infty$. Moreover, they are r.i.~Banach function spaces (possibly up to equivalence of quasinorms) if and only if one of the following three conditions holds:
	\begin{enumerate}
		\item $p = q = 1$,
		\item $p \in (1, \infty)$ and $q \in [1, \infty]$,
		\item $p = q = \infty$.
	\end{enumerate}
\end{example}

\begin{example} \label{ExampleInter&Sum}
	The spaces $L^1 + L^{\infty}$ and $L^{1} \cap L^{\infty}$, which we define via the functionals
	\begin{align*}
		\lVert f \rVert_{L^1 + L^{\infty}} &= \int_0^{1} f^* \: d\lambda, \\
		\lVert f \rVert_{L^{1} \cap L^{\infty}} &= f^*(0) + \int_0^{\infty} f^* \: d\lambda,
	\end{align*}
	are r.i.~Banach function spaces. The above presented norms are equivalent to those obtained via the classical abstract constructions for sums and intersections of Banach spaces, see e.g.~\cite[Chapter~2, Section~6]{BennettSharpley88}.
\end{example}

Finally, let us also provide a definition of an important subspace of an arbitrary quasi-Banach function space, namely the subspace of functions having absolutely continuous quasinorm. 
\begin{definition} \label{DefACqN}
	Let $\lVert \cdot \rVert_X$ be a quasi-Banach function norm and let $X$ be the corresponding quasi-Banach function space. We say that a function $f \in X$ has absolutely continuous quasinorm if it holds that $\lVert f \chi_{E_k} \rVert_X \rightarrow 0$ whenever $E_k$ is a sequence of $\mu$-measurable subsets of $\mathcal{R}$ such that $\chi_{E_k} \rightarrow 0$ $\mu$-a.e. The set of all such functions is denoted $X_a$.
	
	If $X_a = X$, i.e.~every $f \in X$ has absolutely continuous quasinorm, we further say that the space $X$ itself has absolutely continuous quasinorm.
\end{definition}

Let us now move to listing some important properties of r.i.~quasi-Banach function spaces. We begin with two that hold in the wider class of quasi-Banach function spaces:
 
\begin{theorem} \label{TC}
	Let $\lVert \cdot \rVert_X$ be a quasi-Banach function norm and let $X$ be the corresponding quasi-Banach function space. Then $X$ is complete.
\end{theorem}

\begin{theorem} \label{TEQBFS}
	Let $\lVert \cdot \rVert_X$ and $\lVert \cdot \rVert_Y$ be quasi-Banach function norms and let $X$ and $Y$ be the corresponding quasi-Banach function spaces. If $X \subseteq Y$ then also $X \hookrightarrow Y$.
\end{theorem}

Both of these results have been known for a long time in the context of Banach function spaces but they have been only recently extended to quasi-Banach function spaces. Theorem~\ref{TC} has been first obtained by Caetano, Gogatishvili and Opic in \cite{CaetanoGogatishvili16} while Theorem~\ref{TEQBFS} has been proved by the second and third authors in \cite{NekvindaPesa20}. 

Let us now turn to the spaces that are also rearrangement-invariant, as they are the focus of this paper. An important property of r.i.~quasi-Banach function spaces over $([0, \infty), \lambda)$ is that the dilation operator is bounded on those spaces, as stated in the following theorem. This is a classical result in the context of r.i.~Banach function spaces which has been recently extended to r.i.~quasi-Banach function spaces by the second and third authors in \cite{NekvindaPesa20} (for the classical version see for example \cite[Chapter~3, Proposition~5.11]{BennettSharpley88}).

\begin{definition} \label{DDO}
	Let $t \in (0, \infty)$. The dilation operator $D_t$ is defined on $\mathcal{M}([0, \infty), \lambda)$ by the formula
	\begin{equation*}
		D_tf(s) = f(ts),
	\end{equation*}
	where $f \in \mathcal{M}([0, \infty), \lambda)$, $s \in (0, \infty)$.
\end{definition}

\begin{theorem} \label{TDRIS}
	Let $X$ be an r.i.~quasi-Banach function space over $([0, \infty), \lambda)$ and let $t \in (0, \infty)$. Then $D_t: X \rightarrow X$ is a bounded operator.
\end{theorem}

We conclude this section with a result which shows that for r.i.~Banach function spaces, triangle inequality implies local integrability. This result is both useful and interesting from the conceptual point of view, it is therefore rather surprising that it is quite often overlooked.

\begin{proposition} \label{PropP1P5}
	Assume that $(\mathcal{R}, \mu)$ is resonant and let $\lVert \cdot \rVert_X$ be an r.i.~quasi-Banach function norm which satisfies \textup{(P1)} (i.e.~it satisfies the axioms \ref{P1}-\ref{P4}). Then $\lVert \cdot \rVert_X$ also satisfies \textup{(P5)}, i.e.~it is an r.i.~Banach function norm.
\end{proposition}

The proof is based on that of \cite[Chapter~II, Theorem~4.1]{KreinPetunin82}. We will include it here for the sake of completeness, and also because it seems to us that this book is not that well known. 

\begin{proof}
	We will only consider the case when $(\mathcal{R}, \mu)$ is non-atomic as the remaining case of sequences is trivial thanks to \cite[Theorem~3.12]{NekvindaPesa20}. 
	
	Fix $E \subseteq \mathcal{R}$ with $\mu(E) < \infty$. We may assume without any loss of generality that $\mu(E) >0 $. In the first step, we consider non-negative simple functions of a special form. We let $n \in \mathbb{N}$, $n>0$, be arbitrary and assume that we have some pairwise disjoint sets $E_0, \dots, E_{n-1} \subseteq E$ such that we have $\mu(E_i) = \mu(E_j)$ for all pairs $i, j$. We denote this shared value by $\delta$ and also put $\widetilde{E} = \bigcup_{i=1}^{n-1} E_i$. Clearly, $\mu(\widetilde{E}) = n\delta \leq \mu(E)$. We further assume that the sets are chosen in such a way that $\mu(\widetilde{E}) \geq \frac{\mu(E)}{2}$. We may then establish the following lower bound on $\lVert \chi_{\widetilde{E}} \rVert_{X}$:
	
	Since $(\mathcal{R}, \mu)$ is non-atomic, there is some $F \subseteq \widetilde{E}$ satisfying $\mu(F) = \frac{\mu(E)}{2}$ (using the classical Sierpiński theorem). It follows that $\chi_F$ and $\chi_{E \setminus F}$ are equimeasurable, hence we obtain from the rearrangement-invariance of $\lVert \cdot \rVert_X$ and its properties \ref{P1} and \ref{P2} that
	\begin{equation} \label{PropP1P5:01}
		\lVert \chi_E \rVert_X \leq \lVert \chi_F \rVert_X + \lVert \chi_{E \setminus F} \rVert_X = 2 \lVert \chi_F \rVert_X \leq 2 \lVert \chi_{\widetilde{E}} \rVert_X.
	\end{equation}	
		
	Let now $\varphi$ be a simple function of the form
	\begin{equation*}
		\varphi = \sum_{i=0}^{n-1} \alpha_i \chi_{E_i}
	\end{equation*}
	for some numbers $\alpha_i \in [0, \infty)$; this is the special case we consider in this first step of the proof. We further introduce the derived functions $\varphi_k$, for $k \in \{0, \dots, n-1\}$, given by the formulas
	\begin{equation*}
		\varphi_k = \sum_{i=0}^{n-1} \alpha_i \chi_{E_{[(i+k) \mod n ]}}.
	\end{equation*}
	To provide some intuition, the functions $\varphi_k$ are created from $\varphi = \varphi_0$ by a process that can be vaguely described as ``rotating the coefficients over the sets''. It follows that the functions $\varphi_k$ are all equimeasurable and that
	\begin{equation*}
		\sum_{k = 0}^{n-1} \varphi_k = \chi_{\widetilde{E}} \sum_{i = 0}^{n-1} \alpha_i.
	\end{equation*}
	Hence, using the assumptions that $\lVert \cdot \rVert_X$ satisfies \ref{P1} and \ref{P4} and that it is rearrangement-invariant, together with \eqref{PropP1P5:01}, we get
	\begin{equation} \label{PropP1P5:1}
		\int_E \varphi \: d\mu = \int_{\widetilde{E}} \varphi \: d\mu = \left ( \sum_{i=0}^{n-1}  \delta \alpha_i \right ) \frac{\lVert \chi_{\widetilde{E}} \rVert_X}{\lVert \chi_{\widetilde{E}} \rVert_X} = \frac{\delta}{\lVert \chi_{\widetilde{E}} \rVert_X} \left \lVert \sum_{k=0}^{n-1} \varphi_k \right \rVert_X \leq \frac{2\mu(E)}{\lVert \chi_E \rVert_X} \left \lVert \varphi \right \rVert_X.
	\end{equation}
	Note that the constant $\frac{2\mu(E)}{\lVert \chi_E \rVert_X}$ depends only on the measure of $E$ and the norm of its characteristic function.
	
	The next step is to extend the estimate \eqref{PropP1P5:1} to general non-negative simple functions supported on the still fixed set $E$. Consider thus the simple function $\psi$ of the form
	\begin{equation*}
		\psi = \sum_{i=0}^{N-1} \alpha_i \chi_{E_i},
	\end{equation*}
	where $N \in \mathbb{N}$ and $\alpha_i \in [0, \infty)$ are arbitrary and $E_i$ are now arbitrary pairwise disjoint subsets of $E$ (in contrast to the previous step where they all had the same measure). We may assume without loss of generality that
	\begin{equation*}
		\bigcup_{i=0}^{N-1} E_i = E.
	\end{equation*}
	because otherwise we consider
	\begin{equation*}
		\widetilde{\psi} = \psi + 0 \chi_{E \setminus \bigcup_{i=0}^{N-1} E_i},
	\end{equation*}
	instead of $\psi$, which is the same function just written in a way that satisfies this extra requirement. We shall construct a sequence $\varphi_n$ of non-negative simple functions supported on $E$ that are of the form considered in the previous step and that satisfy $\varphi_n \uparrow \psi$ $\mu$-a.e. Then, by applying (in this order) the lattice property \ref{P2} of $\lVert \cdot \rVert_X$ on the right-hand side of \eqref{PropP1P5:1} and the monotone convergence theorem on its left-hand side, we obtain that \eqref{PropP1P5:1} also holds for $\psi$.
	
	The construction proceeds as follows. For each $n \in \mathbb{N}$, we put $\delta_n = 2^{-n}$. For each fixed $E_i$, we find by induction a sequence of partial coverings by sets of measure $\delta_n$. Each of the individual partial coverings will also be constructed by (finite) induction.
	
	Fix $i$ and assume that $\mu(E_i)>0$, otherwise there is nothing to prove. Let $n=0$ and let $K_{i,0} = \lfloor \mu(E_i) \rfloor$ denote the integer part of $\mu(E_i)$. By applying the classical Sierpiński theorem $K_{i,0}$ times, we obtain pairwise disjoint sets $E_{i,0,k}$, for $k \in \{1, \dots, K_{i,0}\}$ such that $\mu(E_{i,0,k}) = \delta_0 = 1$ for all $k$. We further obtain the set 
	\begin{equation*}
		\widehat{E}_{i,0} = E \setminus \left ( \bigcup_{k=1}^{K_{i,0}} E_{i,0,k} \right )
	\end{equation*}
	with $\mu(\widehat{E}_{i,0}) < \delta_0 = 1$.
	
	We note that it is possible for $K_{i,0}$ to be zero, but that is of no consequence, we simply construct no sets in this step and put $\widehat{E}_{i,0} = E$.
	
	Suppose now that we already have for some $n \in \mathbb{N}$ the pairwise disjoint sets $E_{i,n,k}$, $k \in \{1, \dots, K_{i,n}\}$, of measure $\delta_n$ and the set 
	\begin{equation*}
		\widehat{E}_{i,n} = E \setminus \left ( \bigcup_{k=1}^{K_{i,n}} E_{i,n,k} \right )
	\end{equation*}
	 with $\mu(\widehat{E}_{i,n}) < \delta_n$. For $k \in \{1, \dots, K_{i,n}\}$, we simply apply the classical Sierpiński theorem on $E_{i, n, k}$ to ``divide it in halves'', i.e.~obtain two disjoint sets, $E_{i, n+1, 2k}, E_{i, n+1, 2k+1} \subseteq E_{i, n, k}$, with $\mu(E_{i, n+1, 2k}) = \mu(E_{i, n+1, 2k+1}) = \delta_{n+1}$ ($K_{i,n}$ can possibly be zero, in that case we simply do nothing). As for the remainder set $\widehat{E}_{i,n}$, there are two options:
	 \begin{enumerate}
	 	\item Either $\mu(\widehat{E}_{i,n}) < \delta_{n+1} = \frac{\delta_n}{2}$, then we simply put $K_{i, n+1} = 2K_{i, n}$ and $\widehat{E}_{i,n+1} = \widehat{E}_{i,n}$;
	 	\item or $\mu(\widehat{E}_{i,n}) \geq \delta_{n+1}$, then we put $K_{i, n+1} = 2K_{i, n} +1$ and apply again the classical Sierpiński theorem (once) on $\widehat{E}_{i,n}$ to obtain a set $E_{i, n+1, K_{i,n+1}} \subseteq \widehat{E}_{i,n}$ with $\mu(E_{i, n+1, k}) = \delta_{n+1}$. Finally, we put $\widehat{E}_{i,n+1} = \widehat{E}_{i,n} \setminus E_{i,n+1, K_{i, n+1}}$.
	 \end{enumerate}
	 
	 We observe that it holds in either case that
	 \begin{equation*}
	 	\widehat{E}_{i,n+1} = E \setminus \left ( \bigcup_{k=1}^{K_{i,n+1}} E_{i,n+1,k} \right )
	 \end{equation*}
	 and that $\mu(\widehat{E}_{i,n+1}) < \delta_{n+1}$.
	
	Having constructed the sequences of partial coverings of sets $E_i$, we now define the simple functions $\varphi_n$ by
	\begin{equation*}
		\varphi_n = \sum_{i = 0}^{N-1} \sum_{k = 1}^{K_{i,n}} \alpha_i \chi_{E_{i, n, k}}.
	\end{equation*}
	We further put 
	\begin{equation*}
		\widetilde{E}_n = E \setminus \left( \bigcup_{i=0}^{N-1} \widehat{E}_{i, n} \right).
	\end{equation*}
	We observe, that we have $\mu(E_{i, n, k}) = \delta_n$ for all $i, k$ and that $\varphi_n = \psi \chi_{\widetilde{E}}$. Since $\mu(\widehat{E}_{i,n}) < \delta_n$, we see that $\mu(E \setminus \widetilde{E}_n) \to 0$ as $n \to \infty$. Hence, for $n$ large enough, $\varphi_n$ are of the form considered in the first step. Finally, we have for every $i$ and every $n$ that $\widehat{E}_{i, n+1} \subseteq \widehat{E}_{i, n}$, and thus it clearly follows that $\varphi_n \uparrow \psi$ $\mu$-a.e. This concludes the second step of the proof, i.e.~we have shown that \eqref{PropP1P5:1} holds for arbitrary non-negative simple functions supported on $E$.
	
	Finally, to expand \eqref{PropP1P5:1} to arbitrary $f \in \mathcal{M}_+$, one only needs to consider an arbitrary sequence of non-negative simple functions $\psi_n$, supported on $E$, such that $\psi_n \uparrow f \chi_E$ $\mu$-a.e.~and then repeat verbatim the extension procedure employed in the previous step. This construction is classical but also brief, hence we include it for completeness.
	
	For $n \in \mathbb{N}$ we consider the sets $E_{n,i}$ given for $i \in \{0, \dots, n2^n\}$ by
	\begin{equation*}
		E_{n, i} = \left \{x \in E; \; i 2^{-n} \leq f(x) < (i+1) 2^{-n} \right \}
	\end{equation*}
	and put
	\begin{equation*}
		\psi_n = \sum_{i = 0}^{n 2^n} i 2^{-n} \chi_{E_{n,i}}.
	\end{equation*}
\end{proof}

We would like to note that the proof does not use the Fatou property \textup{(P3)}. Furthermore, the assumption that $(\mathcal{R}, \mu)$ is resonant constitutes only a minor loss in generality, as the measure spaces that do not satisfy this assumption are of little interest in the context of r.i.~quasi-Banach function spaces. On the other hand, most of the technical difficulties of the proof are caused by the relative lack of structure in a general non-atomic measure space. For a non-decreasing function defined on an interval, it would be rather trivial to construct an approximating sequence of simple functions of the special type required for the first step of the proof. Hence, Theorem~\ref{TheoremRepresentation} and Corollary~\ref{CorollaryRepresentation} that we present below (together with the Hardy--Littlewood inequality, Theorem~\ref{THLI}) allows for a much simpler proof.

\subsection{Luxemburg representation theorem}

An extremely important result in the theory of r.i.~Banach function spaces it the Luxemburg representation theorem which we already mentioned in the Introduction: 
\begin{theorem} \label{TheoremLuxemburgRepresentation}
	Let $(\mathcal{R},\mu)$ be resonant and let $\lVert \cdot \rVert_X$ be an r.i.~Banach function norm on $\mathcal{M}(\mathcal{R},\mu)$. Then there is an r.i.~Banach function norm $\lVert \cdot \rVert_{\overline{X}}$ on $\mathcal{M}((0,\mu(\mathcal{R})), \lambda)$ such that for every $f \in \mathcal{M}(\mathcal{R},\mu)$ it holds that $\lVert f \rVert_X=\lVert f^* \rVert_{\overline{X}}$.
\end{theorem}

Since we already discussed the history and importance of this result in the Introduction, let us just repeat that this result was first obtained in \cite{Luxemburg67} and that the standard modern reference is \cite[Chapter~2, Theorem~4.10]{BennettSharpley88}.

\subsection{Associate spaces}
Another important concept is that of an associate space. The detailed study of associate spaces of Banach function spaces can be found in \cite[Chapter 1, Sections 2, 3 and 4]{BennettSharpley88}. We will approach the issue in a slightly more general way. The definition of an associate space requires no assumptions on the functional defining the original space.

\begin{definition} \label{DAS}
	Let $\lVert \cdot \rVert_X: \mathcal{M} \to [0, \infty]$ be some non-negative functional and put
	\begin{equation*}
		X = \{ f \in \mathcal{M}; \lVert f \rVert_X < \infty \}.
	\end{equation*} 
	Then the functional $\lVert \cdot \rVert_{X'}$ defined for $f \in \mathcal{M}$ by 
	\begin{equation*}
		\lVert f \rVert_{X'} = \sup_{g \in X} \frac{1}{\lVert g \rVert_X} \int_\mathcal{R} \lvert f g \rvert \: d\mu, \label{DAS1}
	\end{equation*}
	where we interpret $\frac{0}{0} = 0$ and $\frac{a}{0} = \infty$ for any $a>0$, will be called the associate functional of $\lVert \cdot \rVert_X$ while the set
	\begin{equation*}
		X' = \left \{ f \in \mathcal{M}; \lVert f \rVert_{X'} < \infty \right \}
	\end{equation*}
	will be called the associate space of $X$.
\end{definition}

As suggested by the notation, we will be interested mainly in the case when  $\lVert \cdot \rVert_X$ is at least a quasinorm, but we wanted to indicate that such assumption is not necessary for the definition. In fact, it is not even required for the following result, which is the Hölder inequality for associate spaces.

\begin{theorem} \label{THAS}
	Let $\lVert \cdot \rVert_X: \mathcal{M} \to [0, \infty]$ be some non-negative functional and denote by $\lVert \cdot \rVert_{X'}$ its associate functional. Then it holds for all $f,g \in \mathcal{M}$ that
	\begin{equation*}
		\int_\mathcal{R} \lvert f g \rvert \: d\mu \leq \lVert g \rVert_X \lVert f \rVert_{X'}
	\end{equation*}
	provided that we interpret $0 \cdot \infty = -\infty \cdot \infty = \infty$ on the right-hand side.
\end{theorem}

The convention concerning the products at the end of this theorem is necessary precisely because we put no restrictions on $\lVert \cdot \rVert_X$ and thus there occur some pathological cases which need to be taken care of. Specifically, $0 \cdot \infty = \infty$ is required because we allow $\lVert g \rVert_X = 0$ even for non-zero $g$ while $-\infty \cdot \infty = \infty$ is required because Definition~\ref{DAS} allows $X = \emptyset$ which implies $\lVert f \rVert_{X'} = \sup \emptyset = -\infty$.

In order for the associate functional to be well behaved some assumptions on $\lVert \cdot \rVert_X$ are needed. The following result, due to Gogatishvili and Soudsk{\'y} in \cite{GogatishviliSoudsky14}, provides a sufficient condition for the associate functional to be a Banach function norm.

\begin{theorem} \label{TFA}
	Let $\lVert \cdot \rVert_X : \mathcal{M} \to [0, \infty]$ be a functional that satisfies the axioms \ref{P4} and \ref{P5} from the definition of Banach function spaces and which also satisfies for all $f \in \mathcal{M}$ that $\lVert f \rVert_X$ = $\lVert \, \lvert f \rvert \, \rVert_X$. Then the functional $\lVert \cdot \rVert_{X'}$ is a Banach function norm. In addition, $\lVert \cdot \rVert_X$ is equivalent to a Banach function norm if and only if $\lVert \cdot \rVert_X \approx \lVert \cdot \rVert_{X''}$, where $\lVert \cdot \rVert_{X''}$ denotes the associate functional of $\lVert \cdot \rVert_{X'}$.
\end{theorem}

As a special case, we get that the associate functional of any quasi-Banach function space that also satisfies the axiom \ref{P5} is a Banach function norm. This has been observed earlier in \cite[Remark~2.3.(iii)]{EdmundsKerman00}.

Additionally, if $\lVert \cdot \rVert_X$ is a Banach function norm then in fact $\lVert \cdot \rVert_X = \lVert \cdot \rVert_{X''}$. This is a classical result of Lorenz and Luxemburg, proof of which can be found for example in \cite[Chapter~1, Theorem~2.7]{BennettSharpley88}.

Let us point out that even in the case when $\lVert \cdot \rVert_X$, satisfying the assumptions of Theorem~\ref{TFA}, is not equivalent to any Banach function norm we still have the embedding of the space into its second associate space, as formalised in the following statement. The proof is an easy exercise.

\begin{proposition} \label{PESSAS}
	Let $\lVert \cdot \rVert_X$ satisfy the assumptions of Theorem~\ref{TFA}. Then it holds for all $f \in \mathcal{M}$ that
	\begin{equation*}
		\lVert f \rVert_{X''} \leq  \lVert f \rVert_X,
	\end{equation*}
	where $\lVert \cdot \rVert_{X''}$ denotes the associate functional of $\lVert \cdot \rVert_{X'}$.
\end{proposition}

Finally, we include three simple but useful observations. The first one describes how embeddings of spaces translate to embeddings of the corresponding associate spaces.  We formulate it in its full generality to showcase that it does not require any assumptions on the functionals, but we are of course mostly interested in the case when they are quasi-Banach function norm. The proof is an easy modification of \cite[Chapter~2, Proposition~2.10]{BennettSharpley88}.

\begin{proposition} \label{PEASG}
	Let $\lVert \cdot \rVert_X: \mathcal{M} \to [0, \infty]$ and $\lVert \cdot \rVert_Y: \mathcal{M} \to [0, \infty]$ be two non-negative functionals satisfying that there is a constant $C>0$ such that it holds for all $f \in \mathcal{M}$ that
	\begin{equation*}
		\lVert f \rVert_X \leq C \lVert f \rVert_Y.
	\end{equation*}
	Then the associate functionals $\lVert \cdot \rVert_{X'}$ and $\lVert \cdot \rVert_{Y'}$ satisfy, with the same constant $C$,
	\begin{equation*}
		\lVert f \rVert_{Y'} \leq C \lVert f \rVert_{X'}
	\end{equation*}
	for all $f \in \mathcal{M}$.
\end{proposition}

The second observation is a direct consequence of the previous three statements which states that the second associate space $X''$ of $X$ is the smallest Banach function space containing $X$.
\begin{corollary}
	Let $\lVert \cdot \rVert_X$ be a quasi-Banach function norm satisfying \ref{P5} and let $\lVert \cdot \rVert_Y$ be a Banach function norm such that
	\begin{equation*}
		\lVert \cdot \rVert_Y \lesssim \lVert \cdot \rVert_X.
	\end{equation*}
	Then $\lVert \cdot \rVert_{X''}$ is a Banach function norm and we have
	\begin{equation*}
		\lVert \cdot \rVert_Y \lesssim \lVert \cdot \rVert_{X''} \leq \lVert \cdot \rVert_X.
	\end{equation*}
	Equivalently, the corresponding (quasi-)Banach function spaces satisfy
	\begin{equation*}
		X \hookrightarrow X'' \hookrightarrow Y.
	\end{equation*}
\end{corollary}

The final statement shows, that in the case when the underlying measure space is resonant, the associate functional of an r.i.~quasi-Banach function norm can be expressed in terms of non-increasing rearrangement. The proof is the same as in \cite[Chapter~2, Proposition~4.2]{BennettSharpley88}.

\begin{proposition} \label{PAS}
	Let $\lVert \cdot \rVert_X$ be an r.i.~quasi-Banach function norm over a resonant measure space. Then its associate functional $\lVert \cdot \rVert_{X'}$ satisfies
	\begin{equation*}
		\lVert f \rVert_{X'} = \sup_{g \in X} \frac{1}{\lVert g \rVert_{X}} \int_0^{\infty} f^* g^* \: d\lambda.
	\end{equation*}
\end{proposition}

An obvious consequence of Proposition~\ref{PAS} is that an associate space of an r.i.~quasi-Banach function space (over a resonant measure space) is also rearrangement-invariant.

\subsection{Fundamental function and endpoint spaces} \label{SectionFundamentalFunction}
An important tool for understanding the r.i.~Banach function spaces is the fundamental function, which captures some essential behaviour of a given space. Spaces sharing the same fundamental function are in certain sense similar and their differences are therefore necessarily more subtle. Furthermore, any given fundamental function naturally induces three important function spaces: the Lorentz and Marcinkiewicz endpoint spaces, in which we are very much interested in this paper, and also an Orlicz space. All these concepts extend naturally to the wider class of r.i.~quasi-Banach function spaces; however, many of the related theorems do not. We will study the properties of the fundamental function of r.i.~quasi-Banach function spaces in Section~\ref{SectionFundFuncResearch}, whence we present here the relevant definitions and classical results for r.i.~Banach function spaces.

The definitions and statements in this section are adopted from \cite[Chapter~2, Section~5]{BennettSharpley88} and \cite[Sections~7.9 and 7.10]{FucikKufner13} (with only few trivial modifications). As our study in Section~\ref{SectionFundFuncResearch} will be limited to the case when $(\mathcal{R}, \mu)$ is non-atomic, this will also be our standing assumption in this section. This will allow us to present the theory in simpler form and avoid some technicalities that are unnecessary for our purposes. For a more general approach, we recommend the publications mentioned above.

Finally, throughout this section, $\alpha$ will be a number belonging to the interval $(0, \infty]$. Its purpose will be to replace $\mu(\mathcal{R})$ in situations where the terms in question should not depend on any particular choices of of measure spaces.

\begin{definition}
	Let $\lVert \cdot \rVert_X$ be an r.i.~quasi-Banach function norm and $X$ the corresponding r.i.~quasi-Banach function space. We then define the fundamental function $\varphi_X$ of the space $X$ and corresponding to the quasinorm $\lVert \cdot \rVert_X$ for every finite $t$ in the range of $\mu$ by
	\begin{equation*}
		\varphi_X(t) = \lVert \chi_{E_t} \rVert_{X} \text{ for any } E_t \subseteq \mathcal{R} \text{ such that } \mu(E_t) = t.
	\end{equation*}
\end{definition}

This function is well defined as its values do not depend on any particular choices of the sets $E_t$ (because $\lVert \cdot \rVert_X$ is assumed to be rearrangement-invariant). Also, if $X$ and $Y$ are r.i.~quasi-Banach function spaces such that the corresponding fundamental functions $\varphi_X, \varphi_Y$ satisfy $\varphi_X \approx \varphi_Y$ then it is sometimes said that $X$ and $Y$ lie on the same fundamental level (see e.g.~\cite{MusilPick23}).

The following definition is crucial as it characterises which functions are the fundamental function of some r.i.~Banach function space.

\begin{definition} \label{DefinitionQuasiconcave}
	Let $\varphi: [0, \alpha) \to [0, \infty)$. We say that $\varphi$ is quasiconcave on $[0, \alpha)$ if it satisfies the following conditions:
	\begin{enumerate}
		\item $\varphi(t) = 0 \iff t=0$.
		\item $\varphi$ is non-decreasing.
		\item $t \mapsto \frac{\varphi(t)}{t}$ is non-increasing on $(0, \alpha)$.
	\end{enumerate}
	When $\alpha = \infty$ we simply say that $\varphi$ is quasiconcave.
\end{definition}

\begin{proposition} \label{PropositionPropertiesOfFundFunc}
	Let $X$ be an r.i.~Banach function space. Then its fundamental function $\varphi_X$ is quasiconcave on $[0,  \mu(\mathcal{R}))$ and also continuous on $(0,  \mu(\mathcal{R}))$.
\end{proposition}

Proof of this proposition can be found in \cite[Chapter~2, Corollary~5.3]{BennettSharpley88}.

The most natural r.i.~Banach function space associated with any quasiconcave function is the Marcinkiewicz endpoint space. As it turns out, it is also the largest such space.

\begin{definition} \label{DefinitionMarcinkiewicz}
	Let $\varphi: [0, \mu(\mathcal{R}) ) \to [0, \infty)$ be quasiconcave on $[0, \mu(\mathcal{R}))$. We then define the Marcinkiewicz endpoint norm $\lVert \cdot \rVert_{M_{\varphi}}$ for $f \in \mathcal{M}$ by
	\begin{equation*}
		\lVert f \rVert_{M_{\varphi}} = \sup_{t \in [0, \mu(\mathcal{R}) )} \varphi(t) f^{**}(t)
	\end{equation*}
	and the corresponding Marcinkiewicz endpoint space $M_{\varphi}$ by
	\begin{equation*}
		M_\varphi = \left \{ f \in \mathcal{M}; \; \lVert f \rVert_{M_\varphi} < \infty \right\}.
	\end{equation*}
\end{definition}

\begin{theorem} \label{TheoremLargestSpace}
	Let $\varphi: [0, \mu(\mathcal{R}) ) \to [0, \infty)$ be quasiconcave on $[0, \mu(\mathcal{R}))$. Then $\lVert \cdot \rVert_{M_{\varphi}}$ is an r.i.~Banach function norm and its corresponding fundamental function is precisely $\varphi$. Furthermore, if an r.i.~Banach function space $X$ satisfies $\varphi_X \approx \varphi$ then $X \hookrightarrow M_{\varphi}$.
\end{theorem}

Proof of this theorem can be found in \cite[Chapter~2, Theorem~5.13]{BennettSharpley88}.

An attentive reader will notice, that Theorem~\ref{TheoremLargestSpace} and Proposition~\ref{PropositionPropertiesOfFundFunc} together imply that quasiconcave functions are continuous (except possibly at zero) and thus they are left continuous on their entire domain. We now state it explicitly, as it is important for understanding the relationship of quasiconcavity with the weaker notions we introduce in Section~\ref{SectionFundFuncResearch}. We also include only the weaker property of left-continuity, as that is precisely what we will need later.

\begin{corollary} \label{CorollaryLeftCont}
	Let $\varphi: [0,  \alpha ) \to [0, \infty)$ be quasiconcave on $[0, \mu(\mathcal{R}))$. Then it is left-continuous.
\end{corollary}

As we have the largest r.i.~quasi-Banach function space for any given fundamental function $\varphi$, the natural question arises whether there is also the smallest such space. The answer is positive, with the caveat that the space can be defined only for convex functions $\varphi$. Luckily, this is of no issue thanks to the following result (see e.g.~\cite[Chapter~2,  Proposition~5.10]{BennettSharpley88}).

\begin{proposition} \label{PropositionEquivConcave}
	Let $\varphi: [0, \alpha) \to [0, \infty)$ be quasiconcave on $[0, \mu(\mathcal{R}))$. Then there is a concave function $\widetilde{\varphi}$ such that $\varphi \approx \widetilde{\varphi}$.
\end{proposition}

\begin{definition} \label{DefinitionLorentz}
	Let $\varphi: [0, \mu(\mathcal{R}) ) \to [0, \infty)$, vanishing at the origin and positive elsewhere, be non-decreasing and concave. We then define the Lorentz endpoint norm $\lVert \cdot \rVert_{\Lambda_{\varphi}}$ for $f \in \mathcal{M}$ by
	\begin{equation*}
		\lVert f \rVert_{\Lambda_{\varphi}} = \int_{0}^{\infty} f^* \: d\varphi, 
	\end{equation*}
	where the integral on the right-hand side is to be interpreted as a Lebesgue--Stieltjes integral with respect to the non-decreasing function $\varphi$. The corresponding Lorentz endpoint space $\Lambda_{\varphi}$ is then defined as
	\begin{equation*}
		\Lambda_\varphi = \left \{ f \in \mathcal{M}; \; \lVert f \rVert_{\Lambda_\varphi} < \infty \right\}.
	\end{equation*}
\end{definition}

\begin{theorem} \label{TheoremSmallestSpace}
	Let $\varphi: [0, \mu(\mathcal{R}) ) \to [0, \infty)$ be as in Definition~\ref{DefinitionLorentz}. Then $\lVert \cdot \rVert_{\Lambda_{\varphi}}$ is an r.i.~Banach function norm and its corresponding fundamental function is precisely $\varphi$. Furthermore, if an r.i.~Banach function space $X$ satisfies $\varphi_X \approx \varphi$ then $\Lambda_{\varphi} \hookrightarrow X$.
\end{theorem}

Proof of this theorem can be found in \cite[Chapter~2, Theorem~5.13]{BennettSharpley88}.

It is clear that the endpoint spaces do not change, up to equivalence of norms, if one replaces the defining function $\varphi$ with an equivalent one, as long as the required (quasi)convexity is preserved. As working with equivalent functions is often desirable, we will introduce the following convention that will also allow us to avoid some unnecessary technicalities.

\begin{convention} \label{ConventionEndpoints}
	Let $\varphi: [0, \mu(\mathcal{R}) ) \to [0, \infty)$ be equivalent to a function that is concave and non-decreasing and vanishes at zero while being positive elsewhere. Denote this function $\widetilde{\varphi}$. We then put
	\begin{align*}
		\Lambda_{\varphi} &= \Lambda_{\widetilde{\varphi}}, & M_\varphi = M_{\widetilde{\varphi}}.
	\end{align*}
\end{convention}

Clearly, it follows from Proposition~\ref{PropositionEquivConcave} that this convention can be used for any function that is equivalent to a quasiconcave function. We will use this convention extensively, identifying all the spaces defined via equivalent functions and passing to equivalent quasinorms without explicit mentions.

\subsection{Wiener--Luxemburg amalgam spaces} \label{SectionWLA}

The formulation and/or proofs of some of our result will require the use of the so-called Wiener--Luxemburg amalgam spaces that were introduced by the third author \cite{Pesa22}. Said spaces are an abstract construction that provides a rigorous tool for defining, for a given pair of r.i.~quasi-Banach function spaces, a new space, in which the behaviour of functions ``near zero'' corresponds to that of functions belonging to the first space while the behaviour said functions ``near infinity'' corresponds to that of functions belonging to the second space.

As the related theory is only important for the proof of our statements, we do not wish to delve deep into in now and only present the fundamental definition of Wiener--Luxemburg amalgams that is necessary to understanding the formulation of some of our results.

\begin{definition} \label{DefWL}
	Let $\lVert \cdot \rVert_A$ and $\lVert \cdot \rVert_B$ be r.i.~quasi-Banach function norms over $\mathcal{M}([0, \infty), \lambda)$. We then define the Wiener--Luxemburg quasinorm $\lVert \cdot \rVert_{WL(A, B)}$, for $f \in \mathcal{M}([0, \infty), \lambda)$, by
	\begin{equation}
		\lVert f \rVert_{WL(A, B)} = \lVert f^* \chi_{[0,1]} \rVert_A + \lVert f^* \chi_{(1, \infty)} \rVert_B \label{DefWLN}
	\end{equation}
	and the corresponding Wiener--Luxemburg amalgam space $WL(A, B)$ as
	\begin{equation*}
		WL(A, B) = \{f \in \mathcal{M}([0, \infty), \lambda); \; \lVert f \rVert_{WL(A, B)} < \infty \}.
	\end{equation*}
	
	Furthermore, we will call the first summand in \eqref{DefWLN} the local component of $\lVert \cdot \rVert_{WL(A, B)}$ while the second summand will be called the global component of $\lVert \cdot \rVert_{WL(A, B)}$.
\end{definition}

We would like to point out, that the reason for the restriction to $\mathcal{M}([0, \infty), \lambda)$ in \cite{Pesa22} is the lack of a version of the Luxemburg representation theorem that would cover the r.i.~quasi-Banach function spaces. Hence, using the results we present in Section~\ref{SectionRepresentation}, this definition can be easily extended to cover pairs of r.i.~quasi-Banach function spaces over any resonant measure space.

\begin{example} \label{ExampleCap&SumAmalgams}
	The spaces $L^1 + L^{\infty}$ and $L^{1} \cap L^{\infty}$, which were defined in Example~\ref{ExampleInter&Sum}, can be characterised as the following Wiener--Luxemburg amalgam spaces:
	\begin{align*}
		L^1 + L^{\infty} &= WL(L^1, L^{\infty}) \\
		L^{1} \cap L^{\infty} &= WL(L^{\infty}, L^1),
	\end{align*}
	up to equivalence of quasinorms. This characterisation follows from \cite[Corollary~3.12]{Pesa22} together with Theorem~\ref{TEQBFS} (as all four spaces are quasi-Banach function spaces).
\end{example}

\section{Representation theorem}\label{SectionRepresentation}

The purpose of this section is to provide the following generalisation of the extremely important Luxemburg representation theorem (see Theorem~\ref{TheoremLuxemburgRepresentation} for the original result):

\begin{theorem} \label{TheoremRepresentation}
	Assume that $(\mathcal{R},\mu)$ is resonant, let $\lVert \cdot \rVert_X$ be an r.i.~quasi-Banach function norm on $\mathcal{M}(\mathcal{R},\mu)$, and denote its modulus of concavity by $C_X$. Then there is an r.i.~quasi-Banach function norm $\lVert \cdot \rVert_{\overline{X}}$ on $\mathcal{M}([0,\mu(\mathcal{R})), \lambda)$ such that for every $f \in \mathcal{M}(\mathcal{R},\mu)$ it holds that $\lVert f \rVert_X=\lVert f^* \rVert_{\overline{X}}$. Furthermore, $\lVert \cdot \rVert_{\overline{X}}$ has the property \ref{P5} whenever $\lVert \cdot \rVert_X$ does and the modulus of concavity $C_{\overline{X}}$ (of $\lVert \cdot \rVert_{\overline{X}}$) satisfies $C_{\overline{X}} \leq C_X$ when $(\mathcal{R},\mu)$ is non-atomic and $C_{\overline{X}} \leq 4C_X^2$ otherwise. Finally, when $(\mathcal{R},\mu)$ is non-atomic then $\lVert \cdot \rVert_{\overline{X}}$ is uniquely determined.
\end{theorem}

We note, that by Theorem~\ref{TheoremCharResonance}, resonant measure space that is not non-atomic is completely atomic with all atoms having the same measure. As this description is rather cumbersome, we prefer the more elegant formulation above.

For the non-atomic case, will need the following lemma:
\begin{lemma} \label{LMPT}
	Assume  that $(\mathcal{R},\mu)$ is a non-atomic measure space. Then there exists a measure-preserving mapping $\sigma$ from $(\mathcal{R},\mu)$ onto the range of $\mu$ (without $\infty$, if applicable), i.e.~$\mu (\sigma^{-1}(A))=\lambda(A)$ for any Lebesgue measurable $A\subseteq [0,\infty)$ such that $\sup A \leq \mu(\mathcal{R})$.
\end{lemma}

\begin{proof}
	For finite measure spaces this result is a special case of \cite[Proposition~7.4]{BennettSharpley88}. Assume that $\mu(\mathcal{R}) = \infty$. Then there is an increasing sequence $\mathcal{R}_n$ of subsets of $\mathcal{R}$ with finite measure such that
	\begin{equation*}
		\bigcup_{n=1}^{\infty} \mathcal{R}_n = \mathcal{R}.
	\end{equation*}
	Put $\mathcal{R}_0 = \emptyset$, $\widetilde{\mathcal{R}_n} = \mathcal{R}_{n} \setminus \mathcal{R}_{n-1}$, and $I_n = [\mu(\mathcal{R}_{n-1}), \mu(\mathcal{R}_{n})]$. Then, for every $n$, $\widetilde{\mathcal{R}_n}$ is a finite measure space and by \cite[Proposition~7.4]{BennettSharpley88} there is a measure preserving transformation $\sigma_n$ that maps $\widetilde{\mathcal{R}_n}$ onto $I_n$. We now define $\sigma$ on $\widetilde{\mathcal{R}_n}$ by $\sigma(x) = \sigma_n(x)$.
	
	It holds for any Lebesgue measurable $A\subseteq [0,\infty)$ such that $\sup A \leq \mu(\mathcal{R})$ that
	\begin{equation*}
		\sigma^{-1}(A) = \sigma^{-1} \left ( \bigcup_{n=1}^{\infty} A \cap I_n \right ) =\bigcup_{n=1}^{\infty} \sigma_n^{-1} \left ( A \cap I_n \right ),
	\end{equation*}
	where the right-hand side is clearly $\mu$-measurable. Furthermore, the sets on the right-hand side are disjoint, and thus
	\begin{equation*}
		\mu(\sigma^{-1}(A)) = \sum_{n=1}^{\infty} \mu(\sigma_n^{-1} \left ( A \cap I_n \right )) = \sum_{n=1}^{\infty} \lambda \left ( A \cap I_n \right ) = \lambda(A),
	\end{equation*}
	where we also use that the sets $A \cap I_n$ intersect each other only on sets of Lebesgue measure zero.
\end{proof}

\begin{proof}[Proof of Theorem~\ref{TheoremRepresentation}]
	We know from Theorem~\ref{TheoremCharResonance} that resonant measure spaces are either non-atomic or completely atomic with all atoms having the same measure. We treat those spaces separately, beginning with the former.
		
	Assume that $(\mathcal{R},\mu)$ is non-atomic and let $\sigma$ be the measure-preserving transform from Lemma~\ref{LMPT} and consider the operator $T :\mathcal{M}([0,\mu(\mathcal{R})), \lambda) \to \mathcal{M}(\mathcal{R},\mu)$ defined for every $f \in \mathcal{M}([0,\mu(\mathcal{R})), \lambda)$ by
	\begin{equation} \label{TheoremRepresentation:Tna}
		T(f) = f \circ \sigma.
	\end{equation}
	It is immediately obvious that $T$ is linear mapping from $\mathcal{M}([0,\mu(\mathcal{R})), \lambda)$ into $\mathcal{M}(\mathcal{R},\mu)$. Moreover, we have for $f, f_1, f_2, f_n \in \mathcal{M}([0,\mu(\mathcal{R}))$ the following:
	\begin{equation} \label{TheoremRepresentation:1}
		\begin{split}
			&T(\lvert f \rvert )=\lvert f \rvert \circ \sigma =  \lvert f \circ \sigma \rvert =  \lvert T(f) \rvert ,\\
			&0 \leq f_1 \leq f_2 \implies 0 \leq f_1 \circ \sigma \leq f_2 \circ \sigma \implies 0 \leq T(f_1) \leq T(f_2), \\
			&0 \le f_n \uparrow f\ \implies 0 \leq f_n \circ \sigma \uparrow f \circ \sigma \implies 0 \leq T(f_n) \uparrow T(f), \\
			&T(\chi_E) = \chi_{\sigma^{-1}(E)} \text{ for every measurable } E\subseteq [0,\mu(\mathcal{R})),
		\end{split}
	\end{equation}
	and the functions $f$ and $T(f)$ are clearly equimeasurable since $\sigma$ is measure-preserving (see \cite[Chapter~1, Proposition~7.1]{BennettSharpley88}). Finally, $T$ is injective, since $T(f) = 0$ $\mu$-a.e.~implies that $f = 0$ $\lambda$-a.e.~on the range of $\sigma$ and this range is the whole of $[0,\mu(\mathcal{R}))$.
	
	We may now define $\lVert \cdot \rVert_{\overline{X}}$ for every $f \in \mathcal{M}([0,\mu(\mathcal{R})), \lambda)$ by the formula
	\begin{align*}
		&\lVert f \rVert_{\overline{X}} = \lVert T(f) \rVert_X.
	\end{align*}
	
	Properties of $T$ and $\sigma$ now ensure that the axioms satisfied by $\lVert \cdot \rVert_X$ are also satisfied by $\lVert \cdot \rVert_{\overline{X}}$. More specifically, as $T$ is linear and injective, it is clear that this functional is a quasinorm and that $C_X$ is a valid constant for the quasi-triangle inequality, i.e.~that the modulus of concavity of $\lVert \cdot \rVert_{\overline{X}}$ is at most $C_X$. Further, the properties of $T$ listed in \eqref{TheoremRepresentation:1} ensure that $\lVert \cdot \rVert_{\overline{X}}$ satisfies \ref{P2}, \ref{P3}, and \ref{P4} and that $\lVert \, \lvert f \rvert \, \rVert_{\overline{X}} = \lVert f \rVert_{\overline{X}}$ for all $f \in \mathcal{M}([0,\mu(\mathcal{R})), \lambda)$. The rearrangement-invariance of $\lVert \cdot \rVert_{\overline{X}}$ now follows from the fact that when $f_1, f_2 \in \mathcal{M}([0,\mu(\mathcal{R})), \lambda)$ are equimeasurable, then so are $T(f_1)$ and $T(f_2)$. Finally, when $\lVert \cdot \rVert_X$ satisfies \ref{P5}, then the property that $\sigma$ is measure preserving implies for every $E \subseteq (0, \mu(\mathcal{R}))$ and every  $f \in \mathcal{M}([0,\mu(\mathcal{R})), \lambda)$ that
	\begin{equation*}
		\int_E \lvert f \rvert \: d\lambda = \int_{\sigma^{-1}(E)} \lvert f \circ \sigma \rvert \: d\mu \leq C_{\sigma^{-1}(E)} \lVert T(f) \rVert_X = C_{\sigma^{-1}(E)} \lVert f \rVert_{\overline{X}},
	\end{equation*}
	where $ C_{\sigma^{-1}(E)}$ is the constant from \ref{P5} of $\lVert \cdot \rVert_X$ for the set $\sigma^{-1}(E)$, and thus $\lVert \cdot \rVert_{\overline{X}}$ also satisfies \ref{P5}.
	
	To finish the construction, it only remains to show that we have for every $g \in \mathcal{M}(\mathcal{R}, \mu)$ that  $\lVert g \rVert_X = \lVert g^* \rVert_{\overline{X}}$ but that is immediate, as $g$ and $T(g^*)$ are equimeasurable, because they are both equimeasurable with $g^*$.
	
	To observe that $\lVert \cdot \rVert_{\overline{X}}$ is in this case uniquely determined, it suffices to realise that by the classical Sierpinski theorem, for every non-increasing right continuous function $g \in \mathcal{M}([0,\mu(\mathcal{R})), \lambda)$ there is some function $\widetilde{g} \in \mathcal{M}(\mathcal{R},\mu)$ such that $g = (\widetilde{g})^*$. Hence, when	$\lVert \cdot \rVert_{\overline{X}}$ is as above and $\lVert \cdot \rVert$ is any other r.i.~quasi-Banach function norm on $\mathcal{M}([0,\mu(\mathcal{R})), \lambda)$ that represents  $\lVert \cdot \rVert_X$, then we have for every $f \in \mathcal{M}([0,\mu(\mathcal{R})), \lambda)$ that
	 \begin{equation*}
	 	\lVert f \rVert = \lVert f^* \rVert = \left \lVert \left( \widetilde{f^*} \right)^* \right \rVert = \left \lVert \widetilde{f^*} \right \rVert_X = \left \lVert \left( \widetilde{f^*} \right)^* \right \rVert_{\overline{X}} = \left \lVert f^* \right \rVert_{\overline{X}} = \left \lVert f \right \rVert_{\overline{X}}.
	 \end{equation*}
	
	Assume now that $(\mathcal{R},\mu)$ is completely atomic with all atoms having the same measure which we denote by $\beta \in (0, \infty)$. As the space is assumed to be $\sigma$-finite, it follows that the set $\mathcal{R}$ is at most countable. We thus may, without loss of generality, assume that $(\mathcal{R},\mu)$ is the set $\mathcal{N} = \mathbb{N} \cap [0, \mu(\mathcal{R}))$ equipped with the measure $\nu = \beta m$, the appropriate multiple of the classical counting measure $m$. We then consider the operator $T :\mathcal{M}([0,\nu(\mathcal{N})), \lambda) \to \mathcal{M}(\mathcal{N}, \nu)$, which is defined for every $f \in \mathcal{M}([0,\nu(\mathcal{N})), \lambda)$ by the pointwise formula
	\begin{align} \label{TheoremRepresentation:T}
		T(f)(n) &= \beta^{-1} \int_{\beta n}^{\beta (n+1)} f^* \: d\lambda, & \text{for } n \in \mathcal{N}. 
	\end{align}
	
	Now, if $g \in \mathcal{M}(\mathcal{N}, \nu)$, then it is easy to verify that
	\begin{equation*}
		g^* = \sum_{n \in \mathcal{N}} g^*(n) \chi_{[\beta n, \beta (n+1))}
	\end{equation*}
	and that, consequently, $g$ is equimeasurable with $T(g^*)$. Furthermore, if follows from the properties of the non-increasing rearrangement and those of Lebesgue integral that it holds for all $f, f_1, f_2, f_n \in \mathcal{M}([0,\nu(\mathcal{N})), \lambda)$ that
	\begin{equation*} \label{TheoremRepresentation:2}
	\begin{split}
		&T(\lvert f \rvert ) =  \lvert T(f) \rvert = T(f),\\
		&T(a f  ) =  \lvert a \rvert T(f), \text{ for all } a \in \mathbb{C}\\
		&0 \leq f_1 \leq f_2 \text{ $\lambda$-a.e.} \implies 0 \leq T(f_1) \leq T(f_2) \text{ $\nu$-a.e.}, \\
		&0 \le f_n \uparrow f \text{ $\lambda$-a.e.}\implies 0 \leq T(f_n) \uparrow T(f) \text{ $\nu$-a.e.}, \\
		&f = 0 \text{ $\lambda$-a.e.} \iff T(f) = 0 \text{ $\nu$-a.e.},
	\end{split}
	\end{equation*}
	and that if $f,g \in \mathcal{M}([0,\nu(\mathcal{N})), \lambda)$ are equimeasurable, then $T(f) = T(g)$ on $\mathcal{N}$. We also easily get, that when $E \subseteq (0,\nu(\mathcal{N}))$ satisfies $\lambda(E) \in (0, \infty)$ and we define
	\begin{align} \label{TheoremRepresentation:3}
		{E}_{-} &= \{n \in \mathcal{N}; \; \beta (n+1) \leq \lambda(E) \} & {E}_{+} &= \{n \in \mathcal{N}; \; \beta n \leq \lambda(E) \},
	\end{align}
	then we have
	\begin{equation*} 
		\chi_{{E}_{-}} \leq T(\chi_E) \leq \chi_{{E}_{+}}
	\end{equation*}
	and
	\begin{equation*}
		\lambda(E) - \beta \leq \nu(E_{-}) \leq \lambda(E) \leq \nu(E_{+}) \leq \lambda(E) + \beta.
	\end{equation*}
	
	We may now define the representation functional $\lVert \cdot \rVert_{\overline{X}}$ is for every $f \in \mathcal{M}([0,\mu(\mathcal{R})), \lambda)$ by the formula
	\begin{equation*}
		\lVert f \rVert_{\overline{X}} =\left \lVert T(f) \right \rVert_X.
	\end{equation*}
	From the properties of $T$ listed above, it immediately follows that this functional is rearrangement-invariant, it satisfies \ref{P2}, \ref{P3}, and \ref{P4}, and that $\lVert \, \lvert f \rvert \, \rVert_{\overline{X}} = \lVert f \rVert_{\overline{X}}$ for all $f \in \mathcal{M}([0,\nu(\mathcal{N})), \lambda)$. Furthermore, to obtain that $\lVert \cdot \rVert_{\overline{X}}$ is a quasinorm, we only need to show that it is subadditive up to a constant, which is done by the following argument:
	
	We start by fixing some arbitrary $f_1, f_2 \in \mathcal{M}([0,\nu(\mathcal{N})), \lambda)$ and performing the following computation for arbitrary fixed $n \in \mathcal{N}$, using the properties of non-increasing rearrangement and a simple change of coordinates:
	\begin{equation*}
	\begin{split}
		\beta^{-1} \int_{\beta n}^{\beta (n+1)} (f_1 + f_2)^*(t) \: dt &\leq \beta^{-1} \int_{\beta n}^{\beta (n+1)} f_1^*\left(\frac{t}{2}\right) + f_2^*\left(\frac{t}{2}\right) \: dt \\
		&= 2\left(\beta^{-1} \int_{\frac{\beta n}{2}}^{\frac{\beta (n+1)}{2}} f_1^*(t) \: dt + \beta^{-1}\int_{\frac{\beta n}{2}}^{\frac{\beta (n+1)}{2}} f_2^*(t) \: dt \right)
	\end{split}
	\end{equation*}
	Hence, if we define an operator $S :\mathcal{M}([0,\nu(\mathcal{N})), \lambda) \to \mathcal{M}(\mathcal{N}, \nu)$ for every $f \in \mathcal{M}([0,\nu(\mathcal{N})), \lambda)$ by the pointwise formula
	\begin{align*}
		S(f)(n) &= \beta^{-1} \int_{\frac{\beta n}{2}}^{\frac{\beta (n+1)}{2}} f^* \: d\lambda,  &\text{for } n \in \mathcal{N}. 
	\end{align*}
	then we have
	\begin{equation} \label{TheoremRepresentation:4}
		T(f_1+f_2) \leq 2 (S(f_1) + S(f_2)).
	\end{equation}
	It remains to show that $\lVert S(f) \rVert_X \lesssim \lVert T(f) \rVert_X$ for every $f \in \mathcal{M}([0,\nu(\mathcal{N})), \lambda)$. To this end, consider for a fixed $f \in \mathcal{M}([0,\nu(\mathcal{N})), \lambda)$ the following sequences:
	\begin{align*}
		\psi_f (n) &= \begin{cases}
			\beta^{-1} \int_{\frac{\beta n}{2}}^{\frac{\beta (n+2)}{2}} f^* \: d\lambda & \text{for } n \in \mathcal{N}, \text{ $n$ even,} \\[0.8em]
			\beta^{-1} \int_{\frac{\beta (n-1)}{2}}^{\frac{\beta (n+1)}{2}} f^* \: d\lambda &\text{for } n \in \mathcal{N}, \text{ $n$ odd.}
		\end{cases} \\
		\psi_{f, \text{ even}} (n) &= \begin{cases}
			\psi_f (n) & \text{for } n \in \mathcal{N}, \text{ $n$ even,} \\
			0 &\text{for } n \in \mathcal{N}, \text{ $n$ odd.}
		\end{cases} \\
		\psi_{f, \text{ odd}} (n) &= \begin{cases}
			0 & \text{for } n \in \mathcal{N}, \text{ $n$ even,} \\
			\psi_f (n) &\text{for } n \in \mathcal{N}, \text{ $n$ odd.}
		\end{cases} \\
		\overline{\psi_{f, \text{ even}}} (n) &= \begin{cases}
			\psi_{f, \text{ even}} (2n) & \text{for } n \in \mathcal{N}, 2n \in \mathcal{N} \\
			0 &\text{for } n \in \mathcal{N}, 2n \notin \mathcal{N}
		\end{cases} \\
		\overline{\psi_{f, \text{ odd}}} (n) &= \begin{cases}
			\psi_{f, \text{ odd}} (2n+1) & \text{for } n \in \mathcal{N}, 2n+1 \in \mathcal{N} \\
			0 &\text{for } n \in \mathcal{N}, 2n+1 \notin \mathcal{N}
		\end{cases}
	\end{align*}
	Clearly, we have $S(f) \leq \psi_f = \psi_{f, \text{ even}} + \psi_{f, \text{ odd}}$ on $\mathcal{N}$ and it is easy to show that $\psi_{f, \text{ even}}$ and $\psi_{f, \text{ odd}}$ are equimeasurable with $\overline{\psi_{f, \text{ even}}}$ and $\overline{\psi_{f, \text{ odd}}}$, respectively, and that we have the estimates $\overline{\psi_{f, \text{ even}}} \leq T(f)$ and $\overline{\psi_{f, \text{ odd}}} \leq T(f)$ on $\mathcal{N}$. It follows, that
	\begin{equation}\label{TheoremRepresentation:5}
		\lVert S(f) \rVert_X \leq \lVert \psi_f \rVert_X \leq C_X \left ( \left \lVert \overline{\psi_{f, \text{ even}}} \right \rVert_X + \left \lVert \overline{\psi_{f, \text{ odd}}} \right \rVert_X  \right ) \leq 2C_X \lVert T(f) \rVert_X.
	\end{equation}
	Combining \eqref{TheoremRepresentation:4} and \eqref{TheoremRepresentation:5} now yields
	\begin{equation*}
		\lVert f_1 + f_2 \rVert_{\overline{X}} \leq 2C_X \left(\lVert S(f_1) \rVert_X + \lVert S(f_2) \rVert_X\right) \leq 4C_X^2 \left(  \lVert f_1 \rVert_{\overline{X}} + \lVert f_2 \rVert_{\overline{X}} \right ),
	\end{equation*}
	which shows that $\lVert \cdot \rVert_{\overline{X}}$ satisfies the triangle inequality up to the constant $4C_X^2$.
	
	Finally, when $\lVert \cdot \rVert_X$ satisfies \ref{P5}, we may use the Hardy--Littlewood inequality (Theorem~\ref{THLI}) and the notation established in \eqref{TheoremRepresentation:3} to compute (for a fixed $E \subseteq (0,\nu(\mathcal{N}))$ with $\lambda(E)<\infty$ and a fixed $f \in \mathcal{M}([0,\nu(\mathcal{N})), \lambda)$)
	\begin{equation*}
		\int_E \lvert f \rvert \: d\lambda \leq \int_0^{\lambda(E)} f^* \: d\lambda \leq \sum_{n=0}^{\nu(E_{+})} \beta \left(  \beta^{-1} \int_{\beta n}^{\beta (n+1)} f^* \: d\lambda \right ) = \int_{E_{+}} T(f) \: d\nu \leq C_{E_{+}} \lVert f \rVert_{\overline{X}},
	\end{equation*}
	where $C_{E_{+}}$ is the constant from the property \ref{P5} of $\lVert \cdot \rVert_X$ for the set $E_{+}$ (which has finite measure).
\end{proof}

We stress that when $(\mathcal{R},\mu)$ is non-atomic, then the quasinorm $\lVert \cdot \rVert_{\overline{X}}$ is uniquely defined, so it is independent of any particular choice of the measure preserving transformation in the construction presented above. On the other hand, in the atomic case the resulting space depends on the particular choice we have made when defining the operator $T$ in \eqref{TheoremRepresentation:T}, which we will now discuss in more detail.

The construction in is a sort of  $\lVert \cdot \rVert_{\overline{X}}$ in the non-atomic case uses a sort of an amalgam approach, which conceptually follows the ideas from \cite{Pesa22}, especially its Appendix~A. The operator $T$ splits the interval $(0,\nu(\mathcal{N}))$ into the appropriate number of subintervals of length $\beta$, takes normalised $L^1$ norm (of the rearranged function) over each of those subintervals, and finally creates a sequence the $n$-th value of which is the norm on the $n$-th interval. This sequence is then measured in the original quasinorm $\lVert \cdot \rVert_X$. Hence, the result is a variant of an r.i.~amalgam space with $X$ being its global component and $L^1$ being its local component. The choice of the global component is of course dictated by the fact that $X$ is the space we want to represent, but since the quasinorm $\lVert \cdot \rVert_X$ is defined on an atomic measure space, it does not posses any local component that we could use. Hence, we must make a choice and we have plenty of options, as it is not hard to see that essentially the same proof would work with $L^1$ being replaced by any r.i.~Banach function space. Clearly, the resulting space depends on this choice. The reason why we have chosen $L^1$ is that in the case when $\lVert \cdot \rVert_X$ is an r.i.~Banach function norm, then with this choice of the local component the quasinorm $\lVert \cdot \rVert_{\overline{X}}$ satisfies for all $f \in \mathcal{M}([0,\nu(\mathcal{N})), \lambda)$
\begin{equation*}
	\lVert f \rVert_{\overline{X}} = \sup_{\substack{g \in X' \\ \lVert g \rVert_{X'} \leq 1}} \int_0^{\infty} f^* g^* \: d\lambda,
\end{equation*}
i.e.~our construction in this case corresponds to the classical one as presented in \cite[Chapter~2, Theorem~4.10]{BennettSharpley88} and the resulting functional is an r.i.~Banach function norm (we leave the details to the reader).

The reverse of Theorem~\ref{TheoremRepresentation} is also true, i.e.~given any r.i.~quasi-Banach function norm over $\mathcal{M}([0,\mu(\mathcal{R})), \lambda)$ we may use it to define an r.i.~quasi-Banach function norm over arbitrary $\sigma$-finite measure spaces.

\begin{proposition} \label{PropositionInverseRepresentation}
	Let $\lVert \cdot \rVert_{Y}$ be an r.i.~quasi-Banach function norm on $\mathcal{M}([0,\alpha), \lambda)$, $\alpha \in (0, \infty]$, and let $(\mathcal{R},\mu)$ be an arbitrary $\sigma$-finite measure space satisfying $\mu(\mathcal{R}) = \alpha$. Then the functional $\lVert \cdot \rVert_{\widetilde{Y}}$ on $\mathcal{M}(\mathcal{R},\mu)$ given for every $g \in \mathcal{M}(\mathcal{R},\mu)$ by
	\begin{equation*}
		\lVert g \rVert_{\widetilde{Y}} = \lVert g^* \rVert_{Y}
	\end{equation*}
	is also an r.i.~quasi-Banach function norm which satisfies \ref{P5} whenever $\lVert \cdot \rVert_{Y}$ does.
\end{proposition}

\begin{proof}
	The rearrangement-invariance of $\lVert \cdot \rVert_{\widetilde{Y}}$ is obvious. That the properties \ref{P2}, \ref{P3}, and \ref{P4} translate from $\lVert \cdot \rVert_{Y}$ to $\lVert \cdot \rVert_{\widetilde{Y}}$ follows immediately from the properties of non-increasing rearrangement. The same is true for \ref{Q1}, with the exception of part~\ref{Q1c}, i.e.~the quasi-triangle inequality, because the non-increasing rearrangement is not subadditive and the Hardy--Littlewood--P{\' o}lya principle that could bypass this limitation does not hold in some quasi-Banach function spaces (see \cite[Lemma~2.25 and Corollary~5.10]{Pesa22}). We must therefore apply Theorem~\ref{TDRIS} which provides the required estimate, albeit at the cost of worsening the constant. Indeed, it follows from this theorem and the properties of the non-increasing rearrangement that it holds for any pair of functions $g_1, g_2 \in \mathcal{M}(\mathcal{R},\mu)$ that 
	\begin{equation*}
		\lVert g_1 + g_2 \rVert_{\widetilde{Y}} = \lVert (g_1 + g_2)^* \rVert_{Y} \leq C_Y (\lVert D_{\frac{1}{2}} g_1^* \rVert_{Y} + \lVert D_{\frac{1}{2}} g_2^* \rVert_{Y}) \leq C_Y \lVert D_{\frac{1}{2}} \rVert (\lVert g_1 \rVert_{\widetilde{Y}} + \lVert g_2 \rVert_{\widetilde{Y}}),
	\end{equation*}
	where $\lVert D_{\frac{1}{2}} \rVert$ is the norm of the operator $D_{\frac{1}{2}}: Y \to Y$.
	
	Finally, if $\lVert \cdot \rVert_{X}$ satisfies \ref{P5} then it follows from the Hardy--Littlewood inequality (Theorem~\ref{THLI}) that it holds for any $E \subseteq \mathcal{R}$, satisfying both $\mu(E) \leq \alpha = \mu(\mathcal{R})$ and $\mu(E) < \infty$, and any $g \in (\mathcal{R},\mu)$ that
	\begin{equation*}
		\int_E \lvert g \rvert \: d\mu \leq \int_0^{\mu(E)} g^* \: d\lambda \leq C_{(0, \mu(E))} \lVert g^* \rVert_{Y} = C_{(0, \mu(E))} \lVert g \rVert_{\widetilde{Y}},
	\end{equation*}
	where $C_{(0, \mu(E))}$ is the constant from \ref{P5} of $\lVert \cdot \rVert_Y$ for the set $(0, \mu(E))$, and thus $\lVert \cdot \rVert_{\widetilde{Y}}$ also satisfies \ref{P5}.
\end{proof}

Note that the only point in the proof where we needed the rearrangement-invariance of $\lVert \cdot \rVert_{Y}$ was when proving the quasi-triangle inequality. We also did not even need resonance, much less non-atomicity.

Considering the two concepts presented above, one would naturally expect them to be mutually inverse, at least in the case when $(\mathcal{R},\mu)$ is non-atomic and everything is uniquely determined. When the measure is not non-atomic, then transforming a function $(0,\mu(\mathcal{R})), \lambda)$ to a function on $(\mathcal{R},\mu)$ (i.e.~a sequence) must clearly cause some loss of information that cannot be retrieved in the reverse step, but one could still hope for the inverse relation to hold at least in the opposite order. Luckily, this intuition is completely true, as can be seen from the following result:

\begin{corollary} \label{CorollaryRepresentation}
	Let $(\mathcal{R},\mu)$ be a resonant measure space and consider the r.i.~quasi-Banach function norms $\lVert \cdot \rVert_X$ on $\mathcal{M}(\mathcal{R},\mu)$ and $\lVert \cdot \rVert_{Y}$ on $\mathcal{M}([0,\mu(\mathcal{R})), \lambda)$. Then
	\begin{align*}
		\lVert \cdot \rVert_X &= \lVert \cdot \rVert_{\widetilde{\overline{X}}},
	\end{align*}
	and thus $\lVert \cdot \rVert_{\overline{X}}$ satisfies \ref{P5} if and only if $\lVert \cdot \rVert_X$ does.
	
	When further $(\mathcal{R},\mu)$ is non-atomic then
	\begin{equation*}
		\lVert \cdot \rVert_Y = \lVert \cdot \rVert_{\overline{\widetilde{Y}}},
	\end{equation*}
	and thus $\lVert \cdot \rVert_{\widetilde{Y}}$ satisfies \ref{P5} if and only if $\lVert \cdot \rVert_Y$ does.
\end{corollary}

\begin{proof}
	Let $g \in \mathcal{M}(\mathcal{R},\mu)$, then
	\begin{equation*}
		\lVert g \rVert_X = \lVert g^* \rVert_{\overline{X}} = \lVert g \rVert_{\widetilde{\overline{X}}}.
	\end{equation*}
	Similarly, when $(\mathcal{R},\mu)$ is non-atomic and $f \in \mathcal{M}([0,\mu(\mathcal{R})), \lambda)$, then
	\begin{equation} \label{CorollaryRepresentation:1}
		\lVert f \rVert_Y = \lVert f^* \rVert_Y = \lVert T(f)^* \rVert_Y = \lVert T(f) \rVert_{\widetilde{Y}} = \lVert T(f)^* \rVert_{\overline{\widetilde{Y}}} = \lVert f^* \rVert_{\overline{\widetilde{Y}}} = \lVert f \rVert_{\overline{\widetilde{Y}}},
	\end{equation}
	where $T$ is the operator used in the proof of Theorem~\ref{TheoremRepresentation}, as defined in \eqref{TheoremRepresentation:Tna}, because $T(f)$ and $f$ are equimeasurable and thus $T(f)^* = f^*$.
	
	The remaining part is clear.
\end{proof}

Let us note that the argument presented in \eqref{CorollaryRepresentation:1} cannot be used for the case of atomic measures, because the operator $T$ as defined in \eqref{TheoremRepresentation:T} is not injective and the functions $f \in \mathcal{M}([0,\mu(\mathcal{R})), \lambda)$ and $T(f) \in \mathcal{M}(\mathcal{R},\mu)$ are generally not equimeasurable.

To conclude this section, we present the following useful corollary of Theorem~\ref{TheoremRepresentation} that strengthens the property \ref{P2} of quasi-Banach function norms that are rearrangement-invariant:

\begin{corollary} \label{Corollary*<*=>||<||}
	Assume that $(\mathcal{R},\mu)$ is resonant, let $\lVert \cdot \rVert_X$ be an r.i.~quasi-Banach function norm on $\mathcal{M}(\mathcal{R},\mu)$ and let $g_1, g_2 \in \mathcal{M}(\mathcal{R},\mu)$. If $g_1^* \leq g_2^*$ then also $\lVert g_1 \rVert_X \leq \lVert g_2 \rVert_X$.
\end{corollary}

\begin{proof}
	Let $\lVert \cdot \rVert_{\overline{X}}$ be the representation spaces as constructed in Theorem~\ref{TheoremRepresentation}. Then we have, using the property \ref{P2} of $\lVert \cdot \rVert_{\overline{X}}$, that
	\begin{equation*}
		\lVert g_1 \rVert_X = \lVert g_1^* \rVert_{\overline{X}} \leq \lVert g_2^* \rVert_{\overline{X}} = \lVert g_2 \rVert_X.
	\end{equation*}
\end{proof}

\section{Fundamental function and endpoint spaces} \label{SectionFundFuncResearch}
This section is dedicated to extending the very useful theory of fundamental function and endpoint spaces to the context of r.i.~quasi-Banach function spaces.

In this section, we will follow two conventions. Firstly, we will restrict ourselves to the case when $(\mathcal{R}, \mu)$ is non-atomic and $\mu(\mathcal{R}) = \infty$. This we believe to be the most interesting case (of resonant measure spaces), as the functions defined on such space can exhibit both local behaviour (i.e.~blow-ups), which is not possible on completely atomic measure spaces, and global behaviour (i.e.~decay), which is not possible on spaces of finite measure. Hence, when covering this case, one is necessarily forced to deal with all the complexities a difficulties of the theory, and this theory then covers all the interesting behaviour which might be absent in the remaining cases. Further advantage is that this restriction allows us to present the theory in a concise way, without having to distinguish several distinct cases, especially when the differences are often only technical in nature.

That being said, we recognise that spaces of finite measure are of significant interest in the applications, whence we explain the modifications that our theory requires to cover this case in Section~\ref{SectionFiniteMeasure}. The section also nicely illustrates our point that the case of infinite measure is most interesting, as it makes apparent that the case of finite measure does not add any new ideas while significant portions of the theory either do not apply or trivialise.

The case when the measure is completely atomic and all the atoms have the same measure (which is the only case of resonant measure spaces beside non-atomic ones by Theorem~\ref{TheoremCharResonance}) we omit entirely as we are unaware of any significant applications that the concepts of fundamental function and endpoint spaces would have in this setting, and because the spaces of sequence are of less interest in general. Furthermore, there is an easy way how to extend many of the results presented below, namely applying the presented theory on the representation space as introduced above (which is always a space over a non-atomic measure space, namely $((0,\nu(\mathcal{N})), \lambda)$, and whose fundamental function is a linear extension of that of the original space) and then transferring the results the back to the original space; we would however like to stress that one has to be careful to check that the transferred properties correspond to the original space and not to the local component introduced in the construction.

The second convention we will follow throughout the section is simply that $a \cdot 0 = 0$ for all $a \in [0, \infty]$.

We start by introducing the weak Marcinkiewicz endpoint space which play a similar role for r.i.~quasi-Banach function spaces to that of the classical Marcinkiewicz endpoint space for r.i.~Banach function spaces. We approach the matter in a broad way to showcase that the concepts appear naturally and work with very weak a~priori assumptions. We would also like to note that we are not the first to consider this type of space, see e.g.~\cite[Section~7.10]{FucikKufner13}, but the studies already present in the literature are neither too deep nor comprehensive, especially when compared to those of the classical endpoint spaces as discussed in Section~\ref{SectionFundamentalFunction}, and many interesting questions have therefore been open until now.

\begin{definition}
	Let $\varphi: [0, \infty) \to [0, \infty)$. We define the functional $\lVert \cdot \rVert_{m_\varphi}$ for $f \in \mathcal{M}$ by
	\begin{align*}
		\lVert f \rVert_{m_\varphi} &= \sup_{t \in [0, \infty)} \varphi(t) f^*(t).
	\end{align*}
\end{definition}

We will be only interested in this functional when the function $\varphi$ satisfies the following minimal assumptions:

\begin{definition}
	A function $\varphi: [0, \infty) \to [0, \infty)$ will be called admissible, provided it is non-decreasing and left-continuous, it satisfies the $\Delta_2$-condition, and it holds that $\varphi(t) = 0$ if and only if $t=0$.
\end{definition}

Note that, for admissible function $\varphi$, the supremum defining $\lVert \cdot \rVert_{m_\varphi}$ could equivalently be taken over $(0, \infty)$. We chose to include the redundant point $0$ because all the functions are naturally defined on $[0, \infty)$ and we thus wanted to reflect this in the definition.

The motivation for the definition of admissible functions is that said functions are precisely those that are fundamental functions of some r.i.~quasi-Banach function space. This follows from the following two results.

\begin{proposition} \label{PropositionAdmiss}
	Let $\varphi: [0, \infty) \to [0, \infty)$ be admissible. Then  $\lVert \cdot \rVert_{m_\varphi}$ is an r.i.~quasi-Banach function norm. Furthermore, the fundamental function corresponding to $\lVert \cdot \rVert_{m_\varphi}$ is $\varphi$.
\end{proposition}

\begin{proof}
	Given the properties of the non-increasing rearrangement, it is clear that $\lVert \cdot \rVert_{m_\varphi}$ satisfies \ref{P2} and \ref{P3} and that it is rearrangement-invariant. Furthermore, since $\varphi$ is assumed to be non-decreasing, we get for every $t \in [0, \infty)$ and every set $E_t$ with $\mu(E_t) = t$ that
	\begin{equation*}
		\lVert \chi_{E_t} \rVert_{m_\varphi} = \sup_{s \in [0, t)} \varphi(s) = \varphi(t)
	\end{equation*}
	and thus the fundamental function corresponding to $\lVert \cdot \rVert_{m_\varphi}$ is precisely $\varphi$. Moreover, the finiteness of $\varphi$ now implies that $\lVert \cdot \rVert_{m_\varphi}$ satisfies \ref{P4}.  
	
	It remains to show that $\lVert \cdot \rVert_{m_\varphi}$ is a quasinorm. That
	\begin{equation*}
		\lVert af \rVert_{m_\varphi} = \lvert a \rvert \lVert f \rVert_{m_\varphi}
	\end{equation*}
	for every $f \in \mathcal{M}(\mathcal{R},\mu)$ and every $a \in \mathbb{C}$ follows immediately from the properties of the non-increasing rearrangement. It is also clear that $\lVert 0 \rVert_{m_\varphi} = 0$. When $f \in \mathcal{M}(\mathcal{R},\mu)$ is not zero $\mu$-a.e.~then there exists $t>0$ such that $f^*(t) > 0$ and thus
	\begin{equation*}
		\lVert f \rVert_{m_\varphi} \geq \varphi(t) f^*(t) > 0.
	\end{equation*}
	Finally, let $f,g \in \mathcal{M}(\mathcal{R},\mu)$, denote by $C_\varphi$ the constant from the $\Delta_2$-condition of $\varphi$, and compute
	\begin{equation*}
	\begin{split}
		\lVert f + g \rVert_{m_\varphi} &\leq \sup_{t \in [0, \infty)} \varphi(t) \left( f^* \left ( \frac{t}{2} \right ) + g^* \left ( \frac{t}{2} \right ) \right ) \\
		&\leq \left ( \sup_{t \in [0, \infty)} C_{\varphi} \varphi \left ( \frac{t}{2} \right ) f^* \left ( \frac{t}{2} \right ) \right ) + \left ( \sup_{t \in [0, \infty)} C_{\varphi} \varphi \left ( \frac{t}{2} \right ) g^* \left ( \frac{t}{2} \right ) \right ) \\
		&\leq C_{\varphi} (\lVert f \rVert_{m_\varphi} + \lVert g \rVert_{m_\varphi}).
	\end{split}
	\end{equation*}
	Hence, $\lVert \cdot \rVert_{m_\varphi}$ is a quasinorm with the modulus of concavity at most $C_{\varphi}$ and the proof is concluded.
\end{proof}

\begin{proposition} \label{Proposition_m}
	Let $\lVert \cdot \rVert_X$ be an r.i.~quasi-Banach function norm. Then its fundamental function $\varphi_X$ is admissible. Whence, $\lVert \cdot \rVert_{m_{\varphi_X}}$ is an r.i.~quasi-Banach function norm. Furthermore, $\lVert \cdot \rVert_{m_{\varphi_X}} \leq \lVert \cdot \rVert_X$.
	
	More generally, whenever $\varphi$ is admissible and $\lVert \cdot \rVert_Y$ is an r.i.~quasi-Banach function norm whose fundamental function $\varphi_Y$ satisfies $\varphi_Y \approx \varphi$ then $\lVert \cdot \rVert_{m_{\varphi}} \lesssim \lVert \cdot \rVert_Y$.
\end{proposition}

\begin{proof}
	That $\varphi_X$ is non-decreasing follows from \ref{P2}, that it is left-continuous from \ref{P3}. The property \ref{P4} ensures that all the values are finite. Finally, \ref{Q1} ensures that $\varphi(t) = 0$ if and only if $t=0$ and that it satisfies the $\Delta_2$-condition, where the constant is at most twice the modulus of concavity of $\lVert \cdot \rVert_X$.
	
	It remains to show that $\lVert \cdot \rVert_{m_{\varphi_X}} \leq \lVert \cdot \rVert_X$, consider therefore the representation quasinorm $\lVert \cdot \rVert_{\overline{X}}$ whose existence was proved in Theorem~\ref{TheoremRepresentation} and let $f \in \mathcal{M}(\mathcal{R},\mu)$ and $t \in [0, \infty)$ be arbitrary. Then clearly $f^*(\cdot) \geq f^*(t) \chi_{[0, t]}(\cdot)$ and thus
	\begin{equation*}
		\lVert f \rVert_X = \lVert f^* \rVert_{\overline{X}} \geq \lVert  f^*(t) \chi_{[0, t]} \rVert_{\overline{X}} = f^*(t) \varphi_X(t),
	\end{equation*}
	where we employ the obvious fact that the fundamental function of a $\lVert \cdot \rVert_X$ coincides with that of its representation quasinorm. The result now follows by taking on the right-hand side the supremum over all relevant $t$.
	
	The more general formulation is more general only in appearance, as it immediately follows from the assumption $\varphi_Y \approx \varphi$ together with the already proved part that
	\begin{equation*}
		\lVert \cdot \rVert_{m_{\varphi}} \approx \lVert \cdot \rVert_{m_{\varphi_Y}} \leq \lVert \cdot \rVert_Y.
	\end{equation*}
\end{proof}

The fact that $\lVert \cdot \rVert_{m_\varphi}$ is the weakest r.i.~quasi-Banach function norm with a given fundamental function motivates the following terminology.

\begin{definition}
	Let $\varphi: [0, \infty) \to [0, \infty)$ be admissible. Then the functional $\lVert \cdot \rVert_{m_\varphi}$ will be called the weak Marcinkiewicz endpoint quasinorm and the corresponding set
	\begin{equation*}
		m_\varphi = \left \{ f \in \mathcal{M}; \; \lVert f \rVert_{m_\varphi} < \infty \right\}
	\end{equation*}
	will be called the weak Marcinkiewicz endpoint space.
\end{definition}

Let us now observe that many classical and naturally appearing r.i.~quasi-Banach function spaces, including weak $L^1$ which is likely the most important member of the class, are in fact weak Marcinkiewicz endpoint spaces for appropriately chosen functions $\varphi$.

\begin{example}
	Let $\varphi(t) = t$ on $[0, \infty)$. Then it follows directly from the respective definitions that $m_{\varphi}$ is the space weak $L^1$, i.e.~$m_{\varphi} = L^{1,\infty}$, with equal quasinorms.
	
	More generally, when $p \in (0, \infty)$ and $\varphi(t) = t^{\frac{1}{p}}$, then $m_{\varphi}$ is the space weak $L^p$, i.e.~$m_{\varphi} = L^{p,\infty}$, with equal quasinorms.
\end{example}

This terminology of course refers to the classical Marcinkiewicz endpoint space (see Section~\ref{SectionFundamentalFunction} for details). Before we study the relationships between these concepts, let us introduce a generalisation of quasiconcavity that has proven useful.

\begin{definition}
	Let $\varphi: [0, \infty) \to [0, \infty)$. We say that $\varphi$ is weakly quasiconcave if it satisfies the following conditions:
	\begin{enumerate}
		\item $\varphi$ is left-continuous.
		\item $\varphi(t) = 0 \iff t=0$.
		\item $\varphi$ is non-decreasing.
		\item $t \mapsto \frac{\varphi(t)}{t}$ is equivalent to a non-increasing function on $(0, \infty)$.
	\end{enumerate}
\end{definition}

It is clear from Corollary~\ref{CorollaryLeftCont} that quasiconcavity implies weak quasiconcavity and it is also easy to check that weakly quasiconcave functions are admissible. In the opposite direction, only the last condition of (weak) quasiconcavity needs to be checked provided that one knows a~priori that a function is admissible. The motivation for weakening this condition is that the weaker version is much easier to verify and the resulting property is still strong enough for our purposes, as seen from the following result, proof of which is virtually the same as that of \cite[Chapter~2, Proposition~5.8]{BennettSharpley88}.

\begin{proposition} \label{PropositionMarcinkiewiczWqc}
	Let $\varphi: [0, \infty) \to [0, \infty)$ be weakly quasiconcave. Then the Marcinkiewicz endpoint space $M_{\varphi}$, defined in the same way as in Definition~\ref{DefinitionMarcinkiewicz}, is an r.i.~Banach function space whose fundamental function is equivalent to  $\varphi$. Specially, $\varphi$ is equivalent to a quasiconcave (and therefore even a concave) function.
\end{proposition}

It is now easy to observe the following characterisation of weakly quasiconcave functions.

\begin{corollary} \label{CorollaryQuasiconcaveEquivalentToConcave}
	Admissible function  $\varphi: [0, \infty) \to [0, \infty)$ is weakly quasiconcave if and only if it is equivalent to some (quasi)concave function that vanishes at zero and is positive elsewhere.
\end{corollary}

\begin{remark} \label{RemarkExNoWqc}
	It is clear that there are r.i.~quasi-Banach function spaces whose fundamental function is not weakly quasiconcave. The most obvious examples are the Lebesgue spaces $L^p$ for $p < 1$.
\end{remark}

We now characterise the relationship between the weak and regular Marcinkiewicz endpoint spaces. We would like to note that this result is already present in literature, see e.g.~\cite[Theorem~5.1]{GogatishviliSoudsky14} or \cite[Proposition~7.10.5]{FucikKufner13}. However, we wish to present our take on the proof, where we tried to highlight while the result is a natural consequence of the properties of said spaces (and make clear which of the properties are relevant).

\begin{theorem} \label{TheoremEquivalenceMm}
	Let $\varphi: [0, \infty) \to [0, \infty)$ be admissible. Then the following three statements are equivalent:
	\begin{enumerate}
		\item  $\lVert \cdot \rVert_{m_\varphi}$ is equivalent to an r.i.~Banach function norm. \label{TheoremEquivalenceMmi}
		\item  $\varphi$ is weakly quasiconcave and we have $m_{\varphi} = M_{\varphi}$ up to equivalence of quasinorms. \label{TheoremEquivalenceMmii}
		\item  $\varphi$ is weakly quasiconcave and it holds that \label{TheoremEquivalenceMmiii}
		\begin{equation}
			\sup_{t \in (0, \infty)} \frac{\varphi(t)}{t} \int_0^{t} \frac{1}{\varphi(s)} \: ds < \infty. \label{TheoremEquivalenceMm:1}
		\end{equation}
	\end{enumerate} 
\end{theorem}

\begin{proof}
	We first show that \ref{TheoremEquivalenceMmi} is equivalent to \ref{TheoremEquivalenceMmii}. When $\varphi$ is weakly quasiconcave and we have $m_{\varphi} = M_{\varphi}$ up to equivalence of quasinorms, then $\lVert \cdot \rVert_{m_\varphi}$ is of course equivalent to the r.i.~Banach function norm $\lVert \cdot \rVert_{M_\varphi}$. On the other hand, when $\lVert \cdot \rVert_{m_\varphi}$ is equivalent to some r.i.~Banach function norm, then the fundamental function corresponding to this norm is equivalent to $\varphi$, which is thus weakly quasiconcave by Corollary~\ref{CorollaryQuasiconcaveEquivalentToConcave}. Hence, we may consider the space $M_\varphi$ (in the sense of both Proposition~\ref{PropositionMarcinkiewiczWqc}, i.e.~defined via $\varphi$, and Convention~\ref{ConventionEndpoints}, i.e.~defined via the equivalent quasiconcave function) and observe from Theorem~\ref{TheoremLargestSpace} that $m_\varphi \hookrightarrow M_{\varphi}$. On the other hand, Proposition~\ref{Proposition_m} ensures the converse embedding, and thus $m_\varphi = M_\varphi$. 
	
	We now show that \ref{TheoremEquivalenceMmii} is equivalent to \ref{TheoremEquivalenceMmiii}. We first note that both statements assume that $\varphi$ is weakly quasiconcave; we shall thus work under this assumption. We further note that, thanks to Theorem~\ref{TEQBFS}, $m_{\varphi} = M_{\varphi}$ holds up to equivalence of quasinorms if and only if they are equal as sets. Furthermore, an arbitrary function $f \in \mathcal{M}$ belongs to $m_{\varphi}$ if and only if it satisfies $f^* \lesssim \frac{1}{\varphi}$. As we always have $M_{\varphi} \hookrightarrow m_{\varphi}$ (see Proposition~\ref{Proposition_m}), it follows from this property and Corollary~\ref{Corollary*<*=>||<||} that $m_{\varphi} = M_{\varphi}$ (in either sense) if and only if $\frac{1}{\varphi} \in \overline{M_{\varphi}}$, which is clearly equivalent to the validity of \eqref{TheoremEquivalenceMm:1}.
\end{proof}

We shall now take a more detailed look on the third condition from the previous theorem and show that appropriately weakened versions of it can be used to characterise some interesting properties of r.i.~quasi-Banach function norms.

\begin{theorem} \label{TheoremEmbeddingsOfm}
	Let $\varphi: [0, \infty) \to [0, \infty)$ be admissible. Then
	\begin{enumerate}
		\item \label{TheoremEmbeddingsOfm:1}
		\begin{equation*}
			m_{\varphi} \hookrightarrow L^1 + L^{\infty} \iff \int_0^{1} \frac{1}{\varphi} \: d\lambda < \infty,
		\end{equation*}
		\item \label{TheoremEmbeddingsOfm:2}
		\begin{equation*}
			L^1 \cap L^{\infty} \hookrightarrow m_{\varphi} \iff \sup_{t \in (2, \infty)} \frac{\varphi(t)}{t}  < \infty.
		\end{equation*}
	\end{enumerate}
	Consequently, $\lVert \cdot \rVert_{m_{\varphi}}$ satisfies \ref{P5} if and only if
	\begin{equation*}
		\int_0^{1} \frac{1}{\varphi} \: d\lambda < \infty.
	\end{equation*}
\end{theorem}

We note that the embedding on the left-hand side of \ref{TheoremEmbeddingsOfm:2} is a necessary condition for the validity of the Hardy--Littlewood--P{\' o}lya principle (by \cite[Theorem~5.9]{Pesa22}).

\begin{proof}
	We first consider the part \ref{TheoremEmbeddingsOfm:1}. Since any function $f \in \mathcal{M}(\mathcal{R}, \mu)$ satisfying $f^* = \frac{1}{\varphi}$ clearly belongs to $m_\varphi$ (and such a function exists thanks to our assumption that $(\mathcal{R}, \mu)$ is non-atomic), it is clear that the left-hand side implies the right-hand side. For the opposite direction, one employs the natural estimate
	\begin{equation*}
		\lVert f \rVert_{L^1 + L^{\infty}} = \int_0^{1} f^* \: d\lambda \leq \Big ( \sup_{t \in [0, 1)} \varphi(t) f^*(t) \Big )  \int_0^{1} \frac{1}{\varphi(t)} \: dt \leq \lVert f \rVert_{m_{\varphi}} \int_0^{1} \frac{1}{\varphi(t)},
	\end{equation*}
	which holds for every  $f \in \mathcal{M}(\mathcal{R}, \mu)$. Furthermore, it follows from \cite[Theorem~5.7]{Pesa22} that the embedding on the left-hand side of \ref{TheoremEmbeddingsOfm:1} is equivalent to $\lVert \cdot \rVert_{m_{\varphi}}$ satisfying \ref{P5} (see also Example~\ref{ExampleCap&SumAmalgams}). 
	
	Consider now the part \ref{TheoremEmbeddingsOfm:2}. We first note that the embedding on the left-hand side is equivalent to the statement that it holds for every $f \in \mathcal{M}(\mathcal{R}, \mu)$ satisfying $f^*(1) < \infty$ that
	\begin{equation} \label{TheoremEmbeddingsOfm:Eq01}
		 \int_1^{\infty} f^* \: d\lambda < \infty \implies f^* \chi_{(1, \infty)} \in\overline{ m_{\varphi}}.
	\end{equation}
	This characterisation follows from \cite[Theorem~5.6]{Pesa22} and \cite[Theorem~5.7]{Pesa22}. Next step is to reformulate the condition on the right hand side. For our computation we will need to use the properties of admissible functions, including the following estimate that holds for every $t \in [1, \infty)$ (we use Proposition~\ref{PropositionD2ImpliesSubadditivity} and a simple calculation):
	\begin{equation}\label{TheoremEmbeddingsOfmEq:1}
		C_{\varphi}^{-1} \varphi(t) - \varphi(1) \leq \varphi(t-1),
	\end{equation}
	where $C_{\varphi}$ is the constant from the $\Delta_2$-condition of $\varphi$. We may now compute the following estimates:
	\begin{equation*}
		\sup_{t \in [1, \infty)} f^*(t) \varphi(t) \geq \sup_{t \in [1, \infty)} f^*(t) \varphi(t-1) \geq C_{\varphi}^{-1} \Big ( \sup_{t \in [1, \infty)} f^*(t) \varphi(t) \Big ) - \varphi(1) f^*(1).
	\end{equation*}
	As clearly
	\begin{equation*}
		\sup_{t \in [1, \infty)} f^*(t) \varphi(t-1) = \lVert f^* \chi_{(1, \infty)} \rVert_{\overline{m_{\varphi}}}
	\end{equation*}
	and $\varphi(1) < \infty$, we have shown that we have the following equivalence for all $f \in \mathcal{M}(\mathcal{R}, \mu)$ satisfying the a~priori assumption $f^*(1) < \infty$:
	\begin{equation*}
		 f^* \chi_{(1, \infty)} \in \overline{m_{\varphi}} \iff \sup_{t \in [1, \infty)} f^*(t) \varphi(t) < \infty.
	\end{equation*}
	
	Assume now that $\varphi$ satisfies the right-hand side of \ref{TheoremEmbeddingsOfm:2}, i.e.
	\begin{equation*}
		\sup_{t \in (2, \infty)} \frac{\varphi(t)}{t}  < \infty,
	\end{equation*}
	and $f \in \mathcal{M}(\mathcal{R}, \mu)$ satisfies $f^*(1)< \infty$ and the left-hand side of \eqref{TheoremEmbeddingsOfm:Eq01}, i.e.
	\begin{equation*}
		\int_1^{\infty} f^* \: d\lambda < \infty.
	\end{equation*}
	Observing that it holds for every $t \in (2, \infty)$ that
	\begin{equation*}
		f^*(t) \leq \frac{1}{t-1} \int_1^{t} f^* \: d\lambda \leq \frac{2}{t} \int_1^{\infty} f^* \: d\lambda,
	\end{equation*}
	we may estimate
	\begin{equation*}
		\sup_{t \in (2, \infty)} f^*(t) \varphi(t) \leq 2 \int_1^{\infty} f^* \: d\lambda \; \sup_{t \in (2, \infty)} \frac{\varphi(t)}{t} < \infty.
	\end{equation*}
	Since
	\begin{equation*}
		\sup_{t \in [1, 2]} f^*(t) \varphi(t) \leq f^*(1) \varphi(2) < \infty,
	\end{equation*}
	we have shown that
	\begin{equation*}
		f^* \chi_{(1, \infty)} \in \overline{m_{\varphi}}.
	\end{equation*}
	Thus, \eqref{TheoremEmbeddingsOfm:Eq01} holds and we have proved that the right-hand side of \ref{TheoremEmbeddingsOfm:2} implies its left-hand side.
	
	As for the opposite direction, since $L^1 \hookrightarrow L^{1, \infty}$, the assumption that $L^1 \cap L^{\infty} \hookrightarrow m_{\varphi}$ together with \cite[Theorem~5.6]{Pesa22} implies that it holds for every $f \in \mathcal{M}(\mathcal{R}, \mu)$ that 
	\begin{equation*}
		\sup_{t \in [0, 1]} t f^*(t) + \sup_{t \in (1, \infty)} \varphi(t-1) f^*(t) \lesssim \int_0^{\infty} f^* \: d\lambda.
	\end{equation*}
	Applying the estimate \eqref{TheoremEmbeddingsOfmEq:1} on the left-hand side, we reach the conclusion that
	\begin{equation} \label{TheoremEmbeddingsOfmEq:2}
		\sup_{t \in [0, \infty)} \widetilde{\varphi}(t) f^*(t) \lesssim f^*(1)\varphi(1) + \int_0^{\infty} f^* \: d\lambda,
	\end{equation}
	where
	\begin{align*} 
		\widetilde{\varphi}(t) &= \begin{cases}
			t &\text{for } t \in [0,1], \\
			\varphi(t) &\text{for } t \in (1, \infty).
		\end{cases}
	\end{align*}
	Assume now that $t_0 \geq 2$. By testing the estimate \eqref{TheoremEmbeddingsOfmEq:2} by the function $\chi_{E_{t_0}}$, where $E_{t_0} \subseteq \mathcal{R}$ satisfies $\mu(E_{t_0}) = t_0$, we obtain
	\begin{equation*}
		\varphi(t_0) = \widetilde{\varphi}(t_0) = \sup_{t \in [0, t_0]} \widetilde{\varphi}(t) \lesssim \varphi(1) + t_0 \lesssim t_0.
	\end{equation*}
	Whence,
	\begin{equation*}
		\sup_{t \in (2, \infty)} \frac{\varphi(t)}{t}  < \infty,
	\end{equation*}
	as desired.
\end{proof}

As the next step in our study of basic properties of the weak Marcinkiewicz endpoint spaces, we present a very useful result that provides a tool for estimating the quasinorm of a given function. This result serves as a sort of replacement for the arguments using Köthe duality that are very common when working with Banach function spaces; it shows, that even though the dual of $m_{\varphi}$ may be trivial (e.g.~$L^{p, \infty}$, $p<1$), the quasinorm can still be estimated via a formula that closely resembles the one used in the arguments via duality.

\begin{theorem}\label{TheoremDualisation}
	Let $\varphi:[0,\infty) \to [0,\infty)$ be admissible, $f\in\mathcal{M}$, and $A \in (0, \infty)$. Denote by $C_{\varphi}$ the constant from the $\Delta_2$-condition of $\varphi$. Then the following statements are true:
	\begin{enumerate}
		\item \label{TheoremDualisation:1} If $\lVert f \rVert_{m_\varphi} \leq A$, then it holds for every $k\in\mathbb{N}$ and every set $E \subseteq \mathcal{R}$ satisfying $\mu(E) \in (0,\infty)$ that there exists a set $E'\subseteq E$ such that $(1-\frac{1}{2^k})\mu(E) \leq \mu(E')$ and
		\begin{equation*}
			\int_{E'} \lvert f \rvert \: d\mu  \leq \frac{\mu(E)}{\varphi(\mu(E))} C_{\varphi}^k A.
		\end{equation*}
		
		\item \label{TheoremDualisation:2} Assume that there is a constant $C \in (0,1]$ such that it holds for every set $E \subseteq \mathcal{R}$ satisfying $\mu(E) \in (0,\infty)$ that there exists a set $E'\subseteq E$ satisfying both $C\mu(E) \leq \mu(E')$ and 
		\begin{equation*}
			\int_{E'} \lvert f \rvert \: d\mu  \leq \frac{\mu(E)}{\varphi(\mu(E))} A.
		\end{equation*}
		Then $\lVert f \rVert_{m_\varphi} \leq C^{-1} A$.
	\end{enumerate}
\end{theorem}

We note that this result has previously been known for the special case of weak Lebesgue spaces $L^{p, \infty}$ (for $p \in (0, 1]$), see e.g.~\cite[Lemma~2.6]{MuscaluSchlag13-2}.

To prove Theorem~\ref{TheoremDualisation}, we will need the following lemma that characterises the weak Marcinkiewicz quasinorm via the distribution function.

\begin{lemma} \label{LemmaEquivCharOfm}
	Let $\varphi: [0, \infty) \to [0, \infty)$ be admissible. Then we have for every $f\in\mathcal{M}$ that
	\begin{equation} \label{LemmaEquivCharOfm:0}
		\sup_{t \in [0, \infty)}f^*(t)\varphi(t) = \sup_{s \in [0, \infty)} s \varphi(f_{*}(s)),
	\end{equation}
	where we interpret $\varphi(\infty) = \lim_{t \to \infty} \varphi(t)$. 
\end{lemma}

\begin{proof}
	We first note that, by our standing convention, $f^*(0) \varphi(0) = 0 = 0 \cdot \varphi(f_{*}(0))$, so both suprema can be without a loss of generality restricted to $(0, \infty)$. Furthermore, the equality clearly holds when $f = 0$ $\mu$-a.e., so we will assume the contrary. We first consider functions $f \in \mathcal{M}$ such that both $f_* < \infty$ and $f^* < \infty$ on $(0, \infty)$. 
	
	We begin by showing that the left-hand side is greater than or equal to the right-hand side. To this end, let $s \in (0, \infty)$ be such that $f_{*}(s) \in (0, \infty)$ and let $\varepsilon \in (0, f_{*}(s))$. Then 
	\begin{equation*}
		f^*(f_{*}(s)-\varepsilon) > s,
	\end{equation*} so we get
	\begin{equation} \label{LemmaEquivCharOfm:1}
		\sup_{t \in (0, \infty)} f^*(t)\varphi(t) \geq f^*(f_{*}(s)-\varepsilon)\varphi(f_{*}(s)-\varepsilon) > s \varphi(f_{*}(s)-\varepsilon).
	\end{equation}
	Now, by sending $\varepsilon \rightarrow 0_+$  and taking the supremum over all $s \in (0, \infty)$ we obtain the desired inequality. 
	
	For the opposite inequality, let $t \in (0, \infty)$ be such that $f^*(t) \in (0, \infty)$ and let $\varepsilon \in (0, f^*(t))$. Then 
	\begin{equation*} 
		f_{*}(f^*(t)-\varepsilon)>t,
	\end{equation*}
	thus we get
	\begin{align} \label{LemmaEquivCharOfm:2}
		\sup_{s \in (0, \infty)} s \varphi(f_{*}(s)) \geq (f^*(t)-\varepsilon)\varphi(f_{*}(f^*(t)-\varepsilon)) \geq (f^*(t)-\varepsilon)\varphi(t).
	\end{align}
	Again, the desired inequality follows by sending $\varepsilon \rightarrow 0_+$ and taking the supremum over all $t \in (0, \infty)$. 
		
	Consider now function $f \in \mathcal{M}$ such that there is a $t_0 \in (0, \infty)$ for which $f^*(t_0) = \infty$. Then the left-hand side of \eqref{LemmaEquivCharOfm:0} is clearly infinite. On the other hand, it follows that $f_*(s) \geq t_0$ for all $s \in (0, \infty)$, and thus the right-hand side is also infinite and the equality holds.
	
	Assume now that we have $f \in \mathcal{M}$ for which $f^* < \infty$ on $(0, \infty)$ but there is an $s \in (0, \infty)$ such that $f_*(s) = \infty$. 	We observe, that for any such $s$ we have that $f^* \geq s$ on $(0, \infty)$. Hence, setting
	\begin{equation*} 
		s_0 = \sup \{s \in (0, \infty); \; f_*(s) = \infty\},
	\end{equation*}
	we conclude that
	\begin{equation} \label{LemmaEquivCharOfm:3}
		\lim_{t \to \infty} f^*(t) \geq s_0.
	\end{equation} 
	Now, if $\lim_{t \to \infty} \varphi(t) = \infty$, then this estimate yields that the left-hand side of \eqref{LemmaEquivCharOfm:0} is infinite, while the infiniteness of the right-hand side follows directly. Hence, it remains only to prove the case when $\lim_{t \to \infty} \varphi(t) < \infty$.
	
	Assume $\lim_{t \to \infty} \varphi(t) = \alpha < \infty$. As $f^* < \infty$ on $(0, \infty)$, the proof that the right-hand side of \eqref{LemmaEquivCharOfm:0} is larger than or equal to the left-hand side is exactly the same as presented above, because the estimate \eqref{LemmaEquivCharOfm:2} can be obtained for all $t \in (0, \infty)$. On the other hand, the estimate \eqref{LemmaEquivCharOfm:1} can be obtained by the method presented above only for $s \in (s_0, \infty)$, because it requires $f_*(s) < \infty$. Considering that $f_*$ is right-continuous and non-increasing while $\varphi$ is left-continuous and non-decreasing, we conclude that
	\begin{equation*}
		\sup_{s \in [s_0, \infty)} s \varphi(f_{*}(s)) \leq \sup_{t \in (0, \infty)}f^*(t)\varphi(t).
	\end{equation*}
	On the other hand, it follows from \eqref{LemmaEquivCharOfm:3} that
	\begin{equation*}
		\sup_{s \in (0, s_0)} s \varphi(f_*(s)) \leq s_0 \alpha \leq \lim_{t \to \infty} f^*(t) \varphi(t) \leq \sup_{t \in (0, \infty)}f^*(t)\varphi(t).
	\end{equation*}
	Combining these two estimates shows that the left-hand side of \eqref{LemmaEquivCharOfm:0} is larger than or equal to the right-hand side.
\end{proof}

\begin{proof}[Proof of Theorem~\ref{TheoremDualisation}]
	In both steps, we will use that
	\begin{equation*}
		 \lVert f \rVert_{m_\varphi} = \sup_{s \in [0, \infty)} s \varphi(f_{*}(s)),
	\end{equation*}
	where we interpret $\varphi(\infty) = \lim_{t \to \infty} \varphi(t)$, which is the characterisation obtained in Lemma~\ref{LemmaEquivCharOfm}. 
	
	We first prove \ref{TheoremDualisation:1}. To this end, fix $k\in\mathbb{N}$, $f\in\mathcal{M}$ such that $\lVert f \rVert_{m_\varphi} \leq A$, and $E\subseteq \mathcal{R}$ such that $\mu(E) \in (0, \infty)$. Furthermore, fix $\varepsilon \in (0, \infty)$ and put $C_{\varepsilon} = C_{\varphi}^k + \varepsilon$. Consider the set $\Omega_{\varepsilon} = \left \{ \lvert f \rvert > \frac{C_{\varepsilon} A}{\varphi(\mu(E))} \right \}$. Then
	\begin{align*}
		\frac{C_{\varepsilon} A}{\varphi(\mu(E))}\varphi(\mu(\Omega_{\varepsilon})) \leq \sup_{t \in [0, \infty)} t\varphi(f_{*}(t)) = \lVert f\rVert_{m_\varphi},
	\end{align*}
	whence 
	\begin{align*}
		\varphi(\mu(\Omega_{\varepsilon})) \leq \frac{\lVert f\rVert_{m_\varphi} \varphi(\mu(E))}{C_{\varepsilon} A} \leq \frac{\varphi(\mu(E))}{C_{\varepsilon}}.
	\end{align*}
	After applying $k$-times the $\Delta_2$-condition of $\varphi$ on this estimate, we obtain that
	\begin{align*}
		\varphi(\mu(\Omega_{\varepsilon})) \leq \frac{C_{\varphi}^k \varphi\left(\frac{\mu(E)}{2^k}\right)}{C_{\varepsilon}} < \varphi\left(\frac{\mu(E)}{2^k}\right),
	\end{align*}
	which implies
	\begin{displaymath}
		\mu(\Omega_{\varepsilon}) \leq \frac{\mu(E)}{2^k},
	\end{displaymath}
	because $\varphi$ is non-decreasing. Furthermore, the sets $\Omega_{\varepsilon}$ grow as $\varepsilon$ decreases. Hence, the set $\Omega = \left \{ \lvert f \rvert > \frac{C_{\varphi}^k A}{\varphi(\mu(E))} \right \} = \bigcup_{\varepsilon > 0} \Omega_{\varepsilon}$ also satisfies
	\begin{displaymath}
		\mu(\Omega) \leq \frac{\mu(E)}{2^k}.
	\end{displaymath}
	Consequently, the set $E' = E \setminus \Omega$ satisfies $\left (1-\frac{1}{2^k} \right )\mu(E) \leq \mu(E')$ and we have
	\begin{align*}
		\int_{E'} \lvert f \rvert \: d\mu \leq \frac{\mu(E') C_{\varphi}^k A}{\varphi(\mu(E))} \leq \frac{\mu(E)}{\varphi(\mu(E))} C_{\varphi}^k A,
	\end{align*}
	as desired.	
	
	We now move on to \ref{TheoremDualisation:2}. Let $C \in (0, 1]$ be as in the assumptions. Let $t \in (0, \infty)$ and put $E_t = \{ \lvert f(x) \rvert >t\}$. Assume for now that $\mu(E_t) < \infty$ and that $E'_t \subseteq E_t$ is as in the assumptions. Then
	\begin{equation} \label{TheoremDualisationEq:1}
		t C \mu(E_t) \leq t \mu(E'_t)  \leq \int_{E'_t} \lvert f \rvert \: d\mu \leq \frac{\mu(E_t)}{\varphi(\mu(E_t))} A,
	\end{equation}
	and thus 
	\begin{equation}  \label{TheoremDualisationEq:2}
		t \varphi(f_*(t)) = t \varphi(\mu(E_t)) \leq \frac{A}{C}.
	\end{equation}
	
	Now, if $f$ is such that $\mu(E_t) < \infty$ for all $t \in (0, \infty)$, we get that 
	\begin{equation*}
		\lVert f\rVert_{m_\varphi} = \sup_{t \in [0, \infty)} t \varphi(f_{*}(t)) \leq \frac {A}{C},
	\end{equation*}
	as desired. On the other hand, if there is a $t \in (0, \infty)$ such that $\mu(E_t) = \infty$, then we may find for every $s \in (0,\infty)$ a set $F_s \subseteq E_t$ satisfying $\mu(F_s) = s$ and for this set the same reasoning that led to \eqref{TheoremDualisationEq:1} produces the estimate
	\begin{equation} \label{TheoremDualisationEq:3}
		\varphi(s) \leq \frac{A}{Ct},
	\end{equation}
	i.e.~$\varphi$ must in this case be a bounded function. We now denote
	\begin{equation*}
		t_0 = \sup \{ t \in (0, \infty); \; \mu(E_t) = \infty\}
	\end{equation*}
	and note that $t_0 < \infty$, because otherwise we would get from \eqref{TheoremDualisationEq:3} that $\varphi$ is constantly equal to zero, which would be a contradiction. It follows that
	\begin{equation*}
		\lim_{s \to \infty} \varphi(s) \leq \frac{A}{Ct_0}
	\end{equation*}
	and therefore
	\begin{equation*}
		\sup_{t \in [0, t_0]} t \varphi(f_{*}(t)) \leq \sup_{t \in [0, t_0]} \frac{At}{Ct_0} \leq \frac {A}{C}.
	\end{equation*}
	Since for the remaining values of $t$ we already have the estimate \eqref{TheoremDualisationEq:2}, we again conclude that
	\begin{equation*}
		\lVert f\rVert_{m_\varphi} \leq \frac {A}{C}.
	\end{equation*}
\end{proof}

We now turn our attention to the space $L^{\infty}$ and show, that its properties are somewhat unique in the class or r.i.~quasi-Banach function spaces. To properly formulate our results, we will use the Wiener--Luxemburg amalgam spaces presented in Section~\ref{SectionWLA}. As we work with an arbitrary non-atomic measure space $\mathcal{M}(\mathcal{R}, \mu)$, we will need to extend their definition to our setting via Theorem~\ref{TheoremRepresentation} and Proposition~\ref{PropositionInverseRepresentation}. To keep our notation reasonable, we will simply write $WL(A,B)$ even when $A, B$ are r.i.~quasi-Banach function spaces over arbitrary $\mathcal{M}(\mathcal{R}, \mu)$, i.e.~in this case we put
\begin{equation*}
	WL(A,B) = \widetilde{WL(\overline{A}, \overline{B})}.
\end{equation*}

\begin{theorem} \label{TheoFFLinftyLoc}
	Let $\lVert \cdot \rVert_X$ be an r.i.~quasi-Banach function norm and let $X$ and $\varphi_X$ be, respectively, the corresponding r.i.~quasi-Banach function space and its fundamental function. Then the following statement are equivalent:
	\begin{enumerate}
		\item $\lim_{t \to 0_+} \varphi_X(t) > 0$, \label{TheoFFLinftyLoc:i}\\
		\item there is a set $E \subseteq \mathcal{R}$ with $\mu(E) < \infty$ such that $\chi_E \notin X_a$, \label{TheoFFLinftyLoc:iib} \\
		\item $X_a = \{0\}$, \label{TheoFFLinftyLoc:ii} \\
		\item $X = WL(L^{\infty}, X)$ up to the equivalence of quasinorms, \label{TheoFFLinftyLoc:iii}
	\end{enumerate}
	where in \ref{TheoFFLinftyLoc:iii} we use the notation of Wiener--Luxemburg amalgam spaces as presented in Definition~\ref{DefWL}.
\end{theorem}

We recall that the set $X_a$ is the subspace of functions having absolutely continuous quasinorm as defined in Definition~\ref{DefACqN}. We would like to note, that for r.i.~Banach function spaces over non-atomic measure spaces the equivalence of \ref{TheoFFLinftyLoc:i}, \ref{TheoFFLinftyLoc:iib}, and \ref{TheoFFLinftyLoc:ii} follows from \cite[Chapter~2, Theorem~5.5]{BennettSharpley88}.

\begin{proof}
	We begin by noting that it is clear that \ref{TheoFFLinftyLoc:iii} implies \ref{TheoFFLinftyLoc:i}. As for the opposite direction, i.e.~\ref{TheoFFLinftyLoc:i}$\implies$\ref{TheoFFLinftyLoc:iii}, we know from \cite[Theorem~5.7]{Pesa22} that $WL(L^{\infty}, X) \hookrightarrow X$, and therefore only the opposite embedding needs to be considered. We further know from \cite[Theorem~5.6]{Pesa22} that this embedding is equivalent to the validity of the following implication (for every $f \in \mathcal{M}$):
	\begin{equation*}
		\lVert f^* \chi_{[0,1]} \rVert_{\overline{X}} < \infty \implies f^*(0) = \lVert f^* \chi_{[0,1]} \rVert_{L^{\infty}} < \infty.
	\end{equation*}
	To prove this implication, we assume that $\lim_{t \to 0_+} f^*(t) = f^*(0) = \infty$ and find some $\varepsilon$ such that $\varphi_X > \varepsilon$ on $(0,\infty)$ (using \ref{TheoFFLinftyLoc:i}). Then for every $n \in \mathbb{N}$ there is a $\delta_n \in (0, \infty)$ such that $f^* > n \chi_{[0, \delta_n]}$ and thus
	\begin{equation*}
		\lVert f^* \chi_{[0,1]} \rVert_{\overline{X}} \geq n \varphi_X(\delta_n) > n \varepsilon.
	\end{equation*}
	Hence, $\lVert f^* \chi_{[0,1]} \rVert_{\overline{X}} = \infty$, as desired.
	
	We shall now prove that \ref{TheoFFLinftyLoc:iib} implies \ref{TheoFFLinftyLoc:i}. When $E$ is as in \ref{TheoFFLinftyLoc:iib}, we may find a sequence $E_n$ of subsets of $\mathcal{R}$ such that $\chi_{E_n} \to 0$ $\mu$-a.e.~while we have $\varepsilon < \lVert \chi_{E_n} \chi_E \rVert_X$ for some $\varepsilon \in (0, \infty)$ and all $n \in \mathbb{N}$. This further implies that $0 < \mu(E_n \cap E) \to 0$. Thence, we obtain the desired conclusion by plugging the values $\mu(E_n \cap E)$ into $\varphi_X$. 
	
	That \ref{TheoFFLinftyLoc:ii} implies \ref{TheoFFLinftyLoc:iib} is trivial. The remaining implication from \ref{TheoFFLinftyLoc:i} to \ref{TheoFFLinftyLoc:ii} follows easily by taking for any characteristic function corresponding to a set $E \subseteq \mathcal{R}$ of finite measure such a sequence of sets $E_n$ for which $\mu(E_n \cap E) > 0$ for all $n \in \mathbb{N}$ (which can always be done by our assumption that $(\mathcal{R}, \mu)$ is non-atomic) and the realising that any $f$ that is not identically equal to zero can be estimated from below by a characteristic function. 
\end{proof}

\begin{theorem} \label{TheoFFLinftyGlob}
	Let $\lVert \cdot \rVert_X$ be an r.i.~quasi-Banach function norm and let $X$ and $\varphi_X$ be, respectively, the corresponding r.i.~quasi-Banach function space and its fundamental function. Then the following statement are equivalent:
	\begin{enumerate}
		\item $\lim_{t \to \infty} \varphi_X(t) < \infty$, \label{TheoFFLinftyGlob:i}\\
		\item there is a function $f \in X$ such that $\lim_{t \to \infty} f^*(t) > 0$, \label{TheoFFLinftyGlob:iib} \\
		\item $\chi_{\mathcal{R}} \in X$, \label{TheoFFLinftyGlob:ii} \\
		\item $X = WL(X, L^{\infty})$ up to the equivalence of quasinorms, \label{TheoFFLinftyGlob:iii}
	\end{enumerate}
	where in \ref{TheoFFLinftyGlob:iii} we use the notation of Wiener--Luxemburg amalgam spaces as presented in Definition~\ref{DefWL}.
\end{theorem}

\begin{proof}
	It is clear that \ref{TheoFFLinftyGlob:iib} and \ref{TheoFFLinftyGlob:ii} are equivalent. Moreover, it follows from \ref{TheoFFLinftyGlob:i} and the property \ref{P3} of $\lVert \cdot \rVert_X$ that 
	\begin{equation*}
		\lVert \chi_{\mathcal{R}} \rVert_X = \lim_{t \to \infty} \varphi_X(t) < \infty
	\end{equation*}
	and thus \ref{TheoFFLinftyGlob:i} implies \ref{TheoFFLinftyGlob:ii}. It is also evident that \ref{TheoFFLinftyGlob:iii} implies \ref{TheoFFLinftyGlob:i}, whence it remains only to show that \ref{TheoFFLinftyGlob:ii} implies \ref{TheoFFLinftyGlob:iii}. We note that we always have $X \subseteq WL(X, L^{\infty})$ (thanks to \cite[Theorem~5.7]{Pesa22}) and for the opposite embedding we employ \cite[Theorem~5.6]{Pesa22} which reduces the problem to showing that every $g \in \mathcal{M}$ such that $g^*(1) < \infty$ satisfies $g^*\chi_{(1,\infty)} \in {\overline{X}}$. But this is clear, since by our assumption
	\begin{equation*}
		g^*\chi_{(1,\infty)} \leq g^*(1) \chi_{[0, \infty)} \in \overline{X}.
	\end{equation*}
\end{proof}

Let us note that it follows from the two preceding results that $L^{\infty}$ is the only r.i.~quasi-Banach function space on its fundamental level.

\begin{corollary} \label{CorollaryFundLevelLinfty}
	Let $\lVert \cdot \rVert_X$ be an r.i.~quasi-Banach function norm and let $X$ and $\varphi_X$ be, respectively, the corresponding r.i.~quasi-Banach function space and its fundamental function. Then $\varphi_X \approx 1$ if and only if $X = L^{\infty}$ up to equivalence of quasinorms.
\end{corollary}

There is a classical result (see e.g.~\cite[Chapter~2, Theorem~6.6]{BennettSharpley88}) that says that every r.i.~Banach function space is sandwiched between the intersection and sum of $L^1$ and $L^{\infty}$. This result has been recently made more precise and general in \cite[Theorem~5.7]{Pesa22} (we have already used this result in this paper a few times), where it has been show that $L^1$ is among r.i.~Banach function spaces the locally largest one and the globally smallest one and that $L^{\infty}$ is in the same context the locally smallest one and the globally largest one. Furthermore, it has been shown that the statements concerning $L^{\infty}$ hold even in the wider context of r.i.~quasi-Banach function spaces while the statements concerning $L^1$ do not. The question then naturally arises whether there is some other space (or a pair of spaces) that would replace $L^1$ and serve as the appropriate local and global endpoints for this wider scale of spaces. Using the theory developed above, we show in the following result that in the local case the answer to this question is negative.

\begin{corollary} \label{CorollaryNoLocLargestSpace}
	For every r.i.~quasi-Banach function space $X$ there is an r.i.~quasi-Banach function space $X_0$ such that $X \hookrightarrow WL(X_0, X)$ and the embedding is sharp.
\end{corollary}

\begin{proof}
	Fix $X$ and let $\varphi_X$ be the corresponding fundamental function. We may assume without loss of generality that 
	\begin{equation*}
		\lim_{t \to 0_+} \varphi_X(t) = 0,
	\end{equation*}
	i.e.~that $X \neq WL(L^{\infty}, X)$ (see Theorem~\ref{TheoFFLinftyLoc}), and that $X = WL(m_{\varphi_X}, X)$, up to equivalence of quasinorms, because if either of those assumption is violated then the construction is trivial.
	
	Consider now the function $\varphi : [0, \infty) \to [0, \infty)$ given by the formula $\varphi = \varphi_X^2$. It is easy to check that $\varphi$ is admissible and that
	\begin{equation*}
		\lim_{t \to 0_+} \frac{\varphi_X(t)}{\varphi(t)} = \infty.
	\end{equation*}
	Thence, $m_{\varphi}$ is an r.i.~Banach function space (by Proposition~\ref{PropositionAdmiss}) and 
	\begin{align*}
		\frac{1}{\varphi} \chi{(0, 1]} &\in m_{\varphi}, & \frac{1}{\varphi} \chi{(0, 1]} \notin m_{\varphi_X},
	\end{align*}
	which when plugged into \cite[Theorem~5.6]{Pesa22} shows that $WL(m_{\varphi}, X) \not \hookrightarrow X$. Finally, it is clear that $\varphi \chi_{[0,1]} \lesssim \varphi_X \chi_{[0,1]}$, so the same result yields the desired embedding $X \hookrightarrow WL(m_{\varphi}, X)$.
\end{proof}

We end this section with the following easy observation that will be useful later.

\begin{proposition} \label{PropositionSandwitchedFundamentalFunction}
	Let $Y_1$ and $Y_2$ be r.i.~quasi Banach function spaces such that their respective fundamental functions $\varphi_{Y_1}$ and $\varphi_{Y_2}$ satisfy $\varphi_{Y_1} \approx \varphi_{Y_2}$. Assume that $X$ is an r.i.~quasi-Banach function space satisfying $Y_1 \hookrightarrow X \hookrightarrow Y_2$. Then the fundamental function $\varphi_X$ of $X$ satisfies $\varphi_X \approx \varphi_{Y_1} \approx \varphi_{Y_2}$.
\end{proposition}

\subsection{Fundamental function of the associate space}

The focus of this section is the behaviour of the fundamental functions of the first and second associate spaces of a given r.i.~quasi-Banach function space whose norm satisfies \ref{P5}. This restriction is of course necessary, because otherwise the question is trivial and not at all interesting. 

The starting point of our examination is the observation, that it follows from the Hölder inequality (Theorem~\ref{THAS}) that 
\begin{align} \label{EstimateFundamentalFunctionAssociateSpace}
	\varphi_{X'}(t) &\geq \frac{t}{\varphi_X(t)}  \text{ for } t \in (0, \infty).
\end{align}
Furthermore, when $X$ is an r.i.~Banach function space then we have equality (see e.g.~\cite[Chapter~2, Theorem~5.2]{BennettSharpley88}). However, this is not a characterisation of r.i.~Banach function spaces, as can readily be observed from the fact that e.g.~$(L^{1, \frac{1}{2}})' = L^{\infty}$ (see e.g.~\cite[Theorem~3.30]{PesaLK}; it also follows from Proposition~\ref{PropositionAssociateSpaceLorentzEndpoint} which we present below). The question then naturally arises whether the spaces for which \eqref{EstimateFundamentalFunctionAssociateSpace} holds with equality can be characterised. It turns out that if one weakens the required equality to equivalence, so that it does not depend on the precise form of the quasinorm, then the answer is positive and the characterisation is a consequence of the following theorem.

\begin{theorem} \label{TheoremFundamentalFunctionX''}
	Let $X$ be an r.i.~quasi-Banach function space whose quasinorm satisfies \ref{P5}, $X''$ its second associate space, and $\varphi_X$ and $\varphi_{X''}$ the respective fundamental functions. Then $\varphi_X \approx \varphi_{X''}$ if and only if $\varphi_X$ is weakly quasiconcave and $X \hookrightarrow M_{\varphi_X}$.
\end{theorem}

Since the associate space of an r.i.~quasi-Banach function space whose norm satisfies \ref{P5} is an r.i.~Banach function space (see Theorem~\ref{TFA}), we know that $\varphi_{X'} = \varphi_{X'''}$, which together with the fact that we have equality in \eqref{EstimateFundamentalFunctionAssociateSpace} when we use it to compute $\varphi_{X'''}$ from $\varphi_{X''}$ yields the following answer to the original question.

\begin{corollary} \label{CorollaryFundamentalFunctionX'}
	Let $X$ be an r.i.~quasi-Banach function space whose quasinorm satisfies \ref{P5}, $X'$ its associate space, and $\varphi_X$ and $\varphi_{X'}$ the respective fundamental functions. Then
	\begin{equation*}
		\varphi_{X'}(t) \approx \frac{t}{\varphi_X(t)}
	\end{equation*}	
	if and only if $\varphi_X$ is weakly quasiconcave and $X \hookrightarrow M_{\varphi_X}$.
\end{corollary}

Before we prove Theorem~\ref{TheoremFundamentalFunctionX''}, we need to check the properties of the function on the right-hand side of \eqref{EstimateFundamentalFunctionAssociateSpace}. It will therefore be useful to introduce some notation for it.

\begin{definition} \label{NotationFundamentalFunctionAssociateSpace}
	Let $\varphi: [0, \infty) \to [0, \infty)$ be admissible. We then put
	\begin{align*}
		\overline{\varphi}(t) &=\begin{cases}
			0 & \text{for } t =0, \\
			\frac{t}{\varphi(t)} & \text{for } t \in (0, \infty).
		\end{cases} 
	\end{align*}
\end{definition}

\begin{proposition} \label{PropositionPhiBar}
	Let  $\varphi: [0, \infty) \to [0, \infty)$ be weakly quasiconcave. Then $\overline{\varphi}$ is equivalent to a weakly quasiconcave function.
\end{proposition}

\begin{proof}
	It is clear that $\overline{\varphi}$ is left-continuous and zero only at zero. It does not have to be non-decreasing, but our assumptions ensure that there is an equivalent non-decreasing function. This function is also zero only at zero and can be chosen to be left-continuous (it is easy to check that this will not break the equivalence in our case). The final requirement then easily follows from the fact that $\frac{1}{\varphi}$ is non-increasing on $(0, \infty)$.
\end{proof}

Proposition~\ref{PropositionPhiBar} implies that when $\varphi$ is weakly quasiconcave, we may consider the spaces $\Lambda_{\overline{\varphi}}$ and $M_{\overline{\varphi}}$ in the sense of Convention~\ref{ConventionEndpoints}.

We now begin to prove Theorem~\ref{TheoremFundamentalFunctionX''} by first treating the following special case that is of independent interest.

\begin{proposition} \label{PropositionAssociateSpaceLorentzEndpoint}
	Let $X$ be an r.i.~quasi-Banach function space whose quasinorm satisfies \ref{P5} and whose fundamental function $\varphi_X$ is weakly quasiconcave. If $X \hookrightarrow \Lambda_{\varphi_X}$ then $X' = M_{\overline{\varphi_X}}$, or equivalently $X'' = \Lambda_{\varphi_X}$, up to equivalence of quasinorms.
\end{proposition}

\begin{proof}
	It is a well known result (see e.g.~\cite[(2.6)]{OpicPick99} and the references therein or \cite[Section~2.5]{PesaLK}) that for any given concave function $\varphi$ we have
	\begin{align} \label{PropositionAssociateSpaceLorentzEndpoint:E1}
		\left( \Lambda_{\varphi} \right)' &= M_{\overline{\varphi}}, & \left( M_{\varphi} \right)' = \Lambda_{\overline{\varphi}},
	\end{align}
	where we use the notation established in Definition~\ref{NotationFundamentalFunctionAssociateSpace}. In light of Convention~\ref{ConventionEndpoints}, Corollary~\ref{CorollaryQuasiconcaveEquivalentToConcave}, and Proposition~\ref{PropositionPhiBar}, this result extends weakly quasiconcave functions $\varphi$. Furthermore, we have from Theorem~\ref{TFA} and Proposition~\ref{PAS} that $X'$ is an r.i.~Banach function space. Hence, it follows from Theorem~\ref{TheoremSmallestSpace}, Proposition~\ref{PEASG}, and Theorem~\ref{TheoremLargestSpace} that
	\begin{equation*}
		M_{\overline{\varphi_X}} \hookrightarrow X' \hookrightarrow M_{\varphi_{X'}}.
	\end{equation*}
	However, since we know from \eqref{EstimateFundamentalFunctionAssociateSpace} that $\overline{\varphi_X} \leq \varphi_{X'}$, it is easy to show that we also have $M_{\varphi_{X'}} \hookrightarrow M_{\overline{\varphi_X}}$. It follows that $X' = M_{\overline{\varphi_X}}$ which is by \eqref{PropositionAssociateSpaceLorentzEndpoint:E1} and Theorem~\ref{TFA} equivalent to $X'' = \Lambda_{\varphi_X}$ (note that clearly $\overline{\overline{\varphi_X}} = \varphi_X$).
\end{proof}

\begin{proof}[Proof of Theorem~\ref{TheoremFundamentalFunctionX''}]
	Since $X''$ is under our assumptions an r.i.~Banach function space (see Theorem~\ref{TFA} and Proposition~\ref{PAS}), it follows from Proposition~\ref{PropositionPropertiesOfFundFunc} and Corollary~\ref{CorollaryQuasiconcaveEquivalentToConcave} that $\varphi_X$ being weakly quasiconcave is a necessary condition for $\varphi_X \approx \varphi_{X''}$. We therefore restrict ourselves to this case. Furthermore, if there exists some function $f \in X \setminus M_{\varphi_X}$, then this function also necessarily satisfies (see Proposition~\ref{PESSAS} and Theorem~\ref{TheoremLargestSpace})
	\begin{equation*}
		f \in X'' \setminus M_{\varphi_X} \; \subseteq \; M_{\varphi_{X''}} \setminus M_{\varphi_X},
	\end{equation*}
	whence $M_{\varphi_{X''}} \neq M_{\varphi_X}$ and this is obviously equivalent to $\varphi_X \not \approx \varphi_{X''}$. We have thus proven the necessity.
	
	As for the sufficiency, consider the space $X \cap \Lambda_{\varphi_X}$ equipped with the usual quasinorm
	\begin{align*}
		\lVert f \rVert_{X \cap \Lambda_{\varphi_X}} &= \max \left \{ \lVert f \rVert_X, \, \lVert f \rVert_{\Lambda_{\varphi_X}} \right \} & \text{for } f \in \mathcal{M}.
	\end{align*}
	It is easy to see that this space is an r.i.~quasi-Banach function spaces whose quasinorm satisfies \ref{P5}, whose fundamental function is equivalent to $\varphi_X$, and that is embedded into both $X$ and $\Lambda_{\varphi_X}$. Hence, our assumptions on $X$, the above recalled relations \eqref{PropositionAssociateSpaceLorentzEndpoint:E1}, Proposition~\ref{PropositionAssociateSpaceLorentzEndpoint}, and Proposition~\ref{PEASG} imply that
	\begin{equation*}
		\Lambda_{\overline{\varphi_X}} \hookrightarrow X' \hookrightarrow \left( X \cap \Lambda_{\varphi_X} \right)' = M_{\overline{\varphi_X}},
	\end{equation*}
	where we again use the notation established in Definition~\ref{NotationFundamentalFunctionAssociateSpace}. Since the fundamental function of both $\Lambda_{\overline{\varphi_X}}$ and $M_{\overline{\varphi_X}}$ is equivalent to $\overline{\varphi_X}$ (see Theorem~\ref{TheoremSmallestSpace} and Theorem~\ref{TheoremLargestSpace} and recall Convention~\ref{ConventionEndpoints}, i.e.~that the actual function defining the spaces is equivalent to $\overline{\varphi_X}$ but it might differ from it), Proposition~\ref{PropositionSandwitchedFundamentalFunction} implies that $\varphi_{X'}$, the fundamental function of $X'$, satisfies
	$\varphi_{X'} \approx \overline{\varphi_X}$. Since $X'$ is an r.i.~Banach function space, as guaranteed by our assumptions and Theorem~\ref{TFA}, it follows that also $\varphi_{X''} = \overline{\varphi_{X'}} \approx \overline{\overline{\varphi_{X}}} = \varphi_X$.
\end{proof}

Finally, we show another result that further illustrates that $L^{\infty}$ is a rather unique space in the class of r.i.~Banach function spaces. This time we show that it is the only space whose second associate space is $L^{\infty}$. Hence, not only is it the only space on its fundamental level (see Corollary~\ref{CorollaryFundLevelLinfty}), but this fundamental level is somehow separated from the rest.

To be as precise as possible, we use the Wiener--Luxemburg amalgam spaces to formulate our result separately for the local and global components, hence we recall the convention concerning amalgams over arbitrary non-atomic measure spaces as established before the formulation of Theorem~\ref{TheoFFLinftyLoc}.

\begin{theorem} \label{TheoremUniqueLinfty}
	Let $X$ an r.i.~quasi-Banach function space whose quasinorm satisfies \ref{P5}. If $X'' = WL(L^{\infty}, X'')$ or $X'' = WL(X'', L^{\infty})$ (up to equivalence of quasinorms), then also $X = WL(L^{\infty}, X)$ or $X = WL(X, L^{\infty})$, respectively (again up to equivalence of quasinorms).
\end{theorem}

We first prove the following easy observation:

\begin{proposition} \label{PropositionAssOf_m}
	Let $\varphi:[0,\infty) \to [0,\infty)$ be admissible. Then the associate functional corresponding to the weak Marcinkiewicz endpoint quasinorm $\lVert \cdot \rVert_{m_\varphi}$ satisfies for all $f \in \mathcal{M}$
	\begin{equation*}
		\lVert f \rVert_{m_{\varphi}'} = \int_{0}^{\infty} \frac{f^*}{\varphi} \: d\lambda.
	\end{equation*}
\end{proposition}

\begin{proof}
	The formula follows directly from Proposition~\ref{PAS} once one realises that we have for every $g \in m_{\varphi}$
	\begin{equation*}
		g^* \leq \frac{\lVert g \rVert_{m_\varphi}}{\varphi} \text{ on } (0, \infty),
	\end{equation*}
	while on the other hand there is a function $g \in \mathcal{M}(\mathcal{R}, \mu)$ satisfying $g^* = \frac{1}{\varphi}$ (because $(\mathcal{R}, \mu)$ is non-atomic) and this function clearly belongs to $m_\varphi$.
\end{proof}

Note that the above proposition holds regardless of whether $\lVert \cdot \rVert_{m_\varphi}$ satisfies \ref{P5}; the behaviour of the formula on the right-hand side is consistent with the characterisation obtained in Theorem~\ref{TheoremEmbeddingsOfm}, in the sense that finiteness of the integral for any non-zero function $f \in \mathcal{M}$ is possible if and only if $\frac{1}{\varphi}$ is integrable near zero. 

We will also need the following lemma to avoid some technical problems in the global part of Theorem~\ref{TheoremUniqueLinfty}.

\begin{lemma}\label{LemmaMaxVarphi}
	Let  $\varphi_0:[0,\infty) \to [0,\infty)$ be admissible and let $m_{\varphi_0}$ be the corresponding weak Marcinkiewicz endpoint space. Consider further the function $\varphi: [0, \infty) \to [0, \infty)$ given by
	\begin{align*}
		\varphi &= \max \{\varphi_0, \chi_{(0, \infty)}\}.
	\end{align*}
	Then  $\varphi$ is admissible and the corresponding weak Marcinkiewicz endpoint space $m_{\varphi}$ satisfies $m_{\varphi} = WL(L^{\infty}, m_{\varphi_0})$.
\end{lemma}

\begin{proof}
	That $\varphi$ is admissible is almost trivial (see Proposition~\ref{PropositionD2Max}), while the necessary local and global embeddings follow from \cite[Theorem~5.6]{Pesa22}.
	
	Indeed, for any $f \in \mathcal{M}$ satisfying $f^*(1) < \infty$ we have
	\begin{equation*}
		\lVert f^* \chi_{(1, \infty)} \rVert_{\overline{m_{\varphi}}} = \max\{\sup_{ t \in \{\varphi_0 < 1\}} f^*(t+1), \sup_{ t \in \{\varphi_0 \geq 1\}} \varphi_0(t) f^*(t+1) \},
	\end{equation*}
	where the first of the suprema on the right-hand side is clearly smaller than $f^*(1)$ and thus finite, while the second is estimated from above by $\lVert f^* \chi_{(1, \infty)} \rVert_{\overline{m_{\varphi_0}}}$. Hence, finiteness of $\lVert f^* \chi_{(1, \infty)} \rVert_{\overline{m_{\varphi_0}}}$ implies that of $\lVert f^* \chi_{(1, \infty)} \rVert_{\overline{m_{\varphi}}}$, while the reverse implication is trivial since $\varphi_0 \leq \varphi$. Finally, the local embedding is trivial as we have $\varphi \chi_{[0,1]} \approx \chi_{[0,1]}$.
\end{proof}

\begin{proof}[Proof of Theorem~\ref{TheoremUniqueLinfty}]
	The local part is easy. When $X'' = WL(L^{\infty}, X'')$, then by combining Proposition~\ref{PESSAS} and \cite[Theorem~5.6]{Pesa22} we obtain that $X \hookrightarrow WL(L^{\infty}, X)$. The essence of the argument is that it follows from our assumption that $X$ is ``locally embedded'' into $L^{\infty}$ while it is of course ``globally embedded'' into itself. Since the opposite embedding always holds as per \cite[Theorem~5.7]{Pesa22}, this proves the local part of the result.
	
	As for the global part, it follows from \cite[Theorem~3.5]{Pesa22} and Theorem~\ref{TFA} that $X'' = WL(X'', L^{\infty})$ if and only if $X' = WL(X', L^{1})$. Denoting by $\varphi_X$ the fundamental function of $X$ and defining a new function $\varphi: [0, \infty) \to [0, \infty)$ by
	\begin{align}
		\varphi &= \max \{\varphi_{X}, \chi_{(0, \infty)}\},
	\end{align}
	we observe from Lemma~\ref{LemmaMaxVarphi} that $\varphi$ is admissible and that the corresponding weak Marcinkiewicz endpoint space $m_{\varphi}$ satisfies $m_{\varphi} = WL(L^{\infty}, m_{\varphi_X})$, where $m_{\varphi_X}$ is of course the weak Marcinkiewicz endpoint space corresponding to $\varphi_X$. Furthermore, it is clear that
	\begin{equation*}
		\int_0^{1} \frac{1}{\varphi} \: d\lambda \leq 1 < \infty,
	\end{equation*}
	that the quasinorm $\lVert \cdot \rVert_{m_{\varphi}}$ therefore satisfies \ref{P5} (see Theorem~\ref{TheoremEmbeddingsOfm}), and that $m_{\varphi}'$, the associate space of $m_{\varphi}$, satisfies $m_{\varphi}' = WL(L^1, (m_{\varphi_X})_i')$, where $(m_{\varphi_X})_i'$ is the integrable associate space of $m_{\varphi_X}$ (see \cite[Definition~4.1]{Pesa22} for the definition of integrable associate spaces and \cite[Corollary~5.5]{Pesa22} for the characterisation of $m_{\varphi}'$).
	
	By combining the above presented observations with Theorem~\ref{TFA}, \cite[Theorem~5.7]{Pesa22}, Proposition~\ref{Proposition_m}, \cite[Proposition~5.11]{Pesa22}, and \cite[Corollary~4.5]{Pesa22} we conclude
	\begin{equation} \label{TheoremUniqueLinfty:E01}
		L^1  \hookrightarrow m_{\varphi}' =  WL(L^1, (m_{\varphi_X})_i') \hookrightarrow WL(L^1, X') = L^{1}.
	\end{equation}
	To provide more insight into how the first embedding in \eqref{TheoremUniqueLinfty:E01} is obtained: since $m_{\varphi}'$ is an r.i.~Banach function space, we know that $L^1$ is ``globally embedded'' into it, while we have observed above that we also have a ``local embedding''.
	
	Hence, we have shown that $m_{\varphi}' = L^1$, which implies boundedness of $\varphi$, as we shall now prove by contradiction. 
	
	Assume that
	\begin{equation*}
		\lim_{t \to \infty} \varphi(t) = \infty,
	\end{equation*}
	consider some arbitrary $\varepsilon \in (0, \infty)$ and find some $K_0 \in (1, \infty)$ sufficiently large such that 
	\begin{align} \label{TheoremUniqueLinfty:E1}
		\varphi(t) &> \frac{1}{\varepsilon} &\text{for all } t \in (K_0, \infty).
	\end{align}
	It is then clear that it holds for every $f \in \mathcal{M}$  that
	\begin{equation} \label{TheoremUniqueLinfty:E1b}
		\int_0^{\infty} \frac{1}{\varphi(t)} f^*(t) \: dt \leq \int_0^{K_0}  f^*(t) \: dt + \int_{K_0}^{\infty} \frac{1}{\varphi(t)} f^*(t) \: dt,
	\end{equation}
	since $\varphi \geq 1$ on $(0, \infty)$.
	
	Now, as we know $m_{\varphi}' = L^{1}$, we may combine Proposition~\ref{PropositionAssOf_m} with \eqref{TheoremUniqueLinfty:E1b} to obtain, that it holds for every $f \in \mathcal{M}$ that
	\begin{equation} \label{TheoremUniqueLinfty:E4}
		\int_0^{\infty} f^* \: d\lambda \leq C_e \int_0^{\infty} \frac{1}{\varphi(t)} f^*(t) \: dt \leq C_e \left ( \int_0^{K_0}  f^*(t) \: dt + \int_{K_0}^{\infty} \frac{1}{\varphi(t)} f^*(t) \: dt \right ),
	\end{equation}
	where $C_e \geq 1$ is some constant for which the embedding $m_{\varphi}' \hookrightarrow L^1$ holds (not necessarily the optimal one); we would like to stress, that this constant does not depend on our choice of $\varepsilon$ and $K_0$. By testing the estimate \eqref{TheoremUniqueLinfty:E4} with functions $f_T = \chi_{E_T}$ where $E_T \subseteq \mathcal{R}$ is an arbitrary set of a given measure $T > K_0$, we get
	\begin{equation*}
		T \leq  C_e \left ( K_0 + \int_{K_0}^{T} \frac{1}{\varphi(t)} \: dt \right ).
	\end{equation*}
	It follows that we have for every $T \in (2 C_e K_0, \infty)$ the estimate 
	\begin{equation*}
		T \leq 2T - 2C_e K_0  \leq 2C_e \int_{K_0}^{T} \frac{1}{\varphi(t)} \: dt,
	\end{equation*}
	and consequently
	\begin{equation*}
		1 \leq 2C_e \frac{1}{T} \int_{K_0}^{T} \frac{1}{\varphi(t)} \: dt.
	\end{equation*}
	However, it follows from \eqref{TheoremUniqueLinfty:E1} that
	\begin{equation*}
		\frac{1}{T} \int_{K_0}^{T} \frac{1}{\varphi(t)} \: dt \leq \frac{1}{T} \int_{0}^{T} \varepsilon \: dt = \varepsilon.
	\end{equation*}
	Since $\varepsilon$ may be chosen arbitrarily small, this leads to a contradiction.
	
	We have thus shown that $\varphi$ is bounded, which is clearly equivalent to $\varphi_X$ being bounded. The desired conclusion now follows directly from Theorem~\ref{TheoFFLinftyGlob}.
\end{proof}

\subsection{The case of finite measure} \label{SectionFiniteMeasure}

In this section, we explain what modifications of the theory presented above are necessary when the underlying measure space has finite measure. We will present all the necessary modifications of definitions and of those results where there are significant differences; when the changes are small and easy we do not repeat the result and trust that the readers will make the modifications, using our commentary and the presented changes to definitions as a guide.

We will, similarly as before, throughout this section use the convention that $\alpha$ is some number belonging to the interval $(0, \infty)$. Its purpose will be to replace $\mu(\mathcal{R})$ in situations where the terms in question should not depend on any particular choices of of measure spaces.

Before we start, we need a restricted version of the $\Delta_2$-condition.
\begin{definition}
	Let $\varphi: [0, \alpha ) \to [0, \infty)$ be non-decreasing. We say that $\varphi$ satisfies $\Delta_2$-condition restricted to $\left [ 0, \alpha \right )$ if there exists a constant $C_{\varphi} \in (0, \infty)$ such that it holds for all $t \in \left [ 0, \frac{ \alpha }{2} \right )$ that
	\begin{equation*}
		\varphi(2t) \leq C_{\varphi} \varphi(t).
	\end{equation*}
\end{definition}

Proposition~\ref{PropositionD2ImpliesSubadditivity} remains valid, one only needs to restrict oneself to those pairs of $s, r$ for which $s+t \in [0, \alpha)$.

We are now in position to define the class of admissible functions and the weak Marcinkiewicz endpoint space.

\begin{definition}
	A function $\varphi: [0, \alpha ) \to [0, \infty)$ will be called admissible on $[0, \alpha)$, provided it is non-decreasing and left-continuous, it satisfies the $\Delta_2$-condition restricted to $\left [ 0, \alpha  \right )$, and it holds that $\varphi(t) = 0$ if and only if $t=0$.
\end{definition}

\begin{definition}
	Let $\varphi: [0, \mu(\mathcal{R})) \to [0, \infty)$. We define the functional $\lVert \cdot \rVert_{m_\varphi}$ for $f \in \mathcal{M}$ by
	\begin{align*}
		\lVert f \rVert_{m_\varphi} &= \sup_{t \in [0, \mu(\mathcal{R}))} \varphi(t) f^*(t).
	\end{align*}
	
	When $\varphi: [0, \mu(\mathcal{R})) \to [0, \infty)$ is admissible on $[0, \mu(\mathcal{R}))$, then the functional $\lVert \cdot \rVert_{m_\varphi}$ will be called the weak Marcinkiewicz endpoint quasinorm and the corresponding set
	\begin{equation*}
		m_\varphi = \left \{ f \in \mathcal{M}; \; \lVert f \rVert_{m_\varphi} < \infty \right\}
	\end{equation*}
	will be called the weak Marcinkiewicz endpoint space.
\end{definition}

Propositions~\ref{PropositionAdmiss} and \ref{Proposition_m} hold as originally presented (except of course the domain of $\varphi$) and the modifications of proofs are trivial. The definition of weak quasiconcavity has to be modified to the following form:

\begin{definition}
	Let $\varphi: [0, \alpha) \to [0, \infty)$. We say that $\varphi$ is weakly quasiconcave on $[0, \alpha)$ if it satisfies the following conditions:
	\begin{enumerate}
		\item $\varphi$ is left-continuous.
		\item $\varphi(t) = 0 \iff t=0$.
		\item $\varphi$ is non-decreasing.
		\item $t \mapsto \frac{\varphi(t)}{t}$ is equivalent to a non-increasing function on $(0, \alpha)$.
	\end{enumerate}
\end{definition}

Proposition~\ref{PropositionMarcinkiewiczWqc} and Corollary~\ref{CorollaryQuasiconcaveEquivalentToConcave} the follow with minimal modifications, Remark~\ref{RemarkExNoWqc} precisely as stated. Theorem~\ref{TheoremEquivalenceMm} requires only one significant change to the statement: the supremum in \eqref{TheoremEquivalenceMm:1} has to be taken over $(0, \mu(\mathcal{R}))$. The proof then remains identical.

From the Theorem~\ref{TheoremEmbeddingsOfm} only the following stump is left, as the most interesting part of the original theorem gets trivialised by the finiteness of the underlying measure:
\begin{theorem} \label{TheoremEmbeddingsOfmFinMeasure}\
	Let $\varphi: [0, \mu(\mathcal{R})) \to [0, \infty)$ be admissible on $[0, \mu(\mathcal{R}))$. Then
	\begin{equation*}
		m_{\varphi} \hookrightarrow L^1 \iff \int_0^{\min\{1, \mu(\mathcal{R})\}} \frac{1}{\varphi} \: d\lambda < \infty.
	\end{equation*}
\end{theorem}

This is caused by the fact that, when working over finite measure, we have $L^1 + L^{\infty} = L^1$, $L^1 \cap L^{\infty} = L^{\infty}$, and trivially $L^{\infty} \hookrightarrow X$ for every r.i.~quasi-Banach function space $X$. The proof of Theorem~\ref{TheoremEmbeddingsOfmFinMeasure} is again virtually the same as that of part~\ref{TheoremEmbeddingsOfm:1} of Theorem~\ref{TheoremEmbeddingsOfm}.

The changes in Lemma~\ref{LemmaEquivCharOfm} are more subtle; the right hand side of the expression does not change, as for the distribution function the range, not the domain, is the property which depends on $\mu(\mathcal{R})$.
\begin{lemma}
	Let $\varphi: [0, \mu(\mathcal{R})) \to [0, \infty)$ be admissible. Then we have for every $f\in\mathcal{M}$ that
	\begin{equation*}
		\sup_{t \in [0, \mu(\mathcal{R}))}f^*(t)\varphi(t) = \sup_{s \in [0, \infty)} s \varphi(f_{*}(s)).
	\end{equation*}
\end{lemma}
The modifications of the proof, however, are simple as they directly follow the modifications of the formulation; the proof is moreover simpler as we now have $f_* \leq \mu(\mathcal{R}) < \infty$. Theorem~\ref{TheoremDualisation} then remains unchanged (besides the obligatory change of the domain of $\varphi$) and its proof differs only in that it is simplified by the fact that in the second step the set $E_t$ is clearly always of finite measure.

Theorem~\ref{TheoFFLinftyLoc} simplifies to the form we present below. The proof needs some modifications, but we leave them to the reader as they are very easy.

\begin{theorem} \label{TheoFFLinftyLocFinMeasure}
	Let $\lVert \cdot \rVert_X$ be an r.i.~quasi-Banach function norm and let $X$ and $\varphi_X$ be, respectively, the corresponding r.i.~quasi-Banach function space and its fundamental function. Then the following statement are equivalent:
	\begin{enumerate}
		\item $\lim_{t \to 0_+} \varphi_X(t) > 0$, \\
		\item there is a set $E \subseteq \mathcal{R}$ with $\mu(E) < \infty$ such that $\chi_E \notin X_a$,  \\
		\item $ X_a = \{0\}$, \\
		\item $X = L^{\infty}$ up to the equivalence of quasinorms.
	\end{enumerate}
\end{theorem}

The statement of Theorem~\ref{TheoFFLinftyGlob} makes no sense in this setting and that of Corollary~\ref{CorollaryFundLevelLinfty} is entirely included in the preceding theorem. Corollary~\ref{CorollaryNoLocLargestSpace} then takes the following form (the proof is virtually unchanged):

\begin{corollary} \label{CorollaryNoLocLargestSpaceFinMeasure}
	For every r.i.~quasi-Banach function space $X$ there is an r.i.~quasi-Banach function space $X_0$ such that $X \hookrightarrow X_0$ and the embedding is sharp.
\end{corollary}

We would like to note that while Theorem~\ref{TheoFFLinftyLocFinMeasure} and Corollary~\ref{CorollaryNoLocLargestSpaceFinMeasure} at first glance appear stronger than their infinite-measure counterparts, this is an illusion. When examined properly, it becomes apparent that the results in both cases convey the same fundamental information, only in the infinite-measure case the formulations need to account for complexity that cannot manifest in the finite-measure case.

Proposition~\ref{PropositionSandwitchedFundamentalFunction} remains unchanged. Theorem~\ref{TheoremFundamentalFunctionX''}, Corollary~\ref{CorollaryFundamentalFunctionX'}, and Proposition~\ref{PropositionAssociateSpaceLorentzEndpoint} hold as presented, with the same proofs, and Proposition~\ref{PropositionPhiBar} requires only cosmetic modifications to correspond with the following modification of Definition~\ref{NotationFundamentalFunctionAssociateSpace}:
\begin{definition}
	Let $\varphi: [0, \alpha) \to [0, \infty)$ be admissible on $[0, \alpha)$. We then put
	\begin{align*}
		\overline{\varphi}(t) &=\begin{cases}
			0 & \text{for } t =0, \\
			\frac{t}{\varphi(t)} & \text{for } t \in (0, \alpha).
		\end{cases} 
	\end{align*}
\end{definition}

Finally, we still have the local part of Theorem~\ref{TheoremUniqueLinfty} as presented below. The proof is identical as before (note that the main difficulty of the proof in the infinite-measure case lied in proving the global part that does not apply in this setting).
\begin{theorem}
	Let $X$ an r.i.~quasi-Banach function space whose quasinorm satisfies \ref{P5}. If $X'' = L^{\infty}$ up to equivalence of quasinorms, then also $X = L^{\infty}$ up to equivalence of quasinorms.
\end{theorem}

\section*{Aknowledgements}
We would like to thank Z.~Mihula, V.~Musil, L.~Pick, T.~Roskovec, L.~Slav{\' i}kov{\' a}, and especially F.~Soudský for their help with the literature.

\bibliographystyle{dabbrv}
\bibliography{bibliography}
\end{document}